\documentclass[10pt,a4paper]{amsart}
\usepackage{amsfonts,amssymb,amscd,amsthm,cite}
\usepackage{hyperref}
\input xy
\xyoption{all}

\usepackage{ifpdf}
\ifpdf %if using pdfLaTeX in PDF mode
  \usepackage[pdftex]{graphicx}
  \DeclareGraphicsExtensions{.pdf,.png,.jpg,.jpeg,.mps}
  \usepackage{pgf}
  \usepackage{tikz}
\else %if using LaTeX or pdfLaTeX in DVI mode
  \usepackage{graphicx}
  \DeclareGraphicsExtensions{.eps,.bmp}
  \usepackage{pgf}
  \usepackage{tikz}
  \usepackage{pstricks}
\fi

\def\ind{\operatorname{ind}}

\def\ord{\operatorname{ord}}

\def\vol{\operatorname{vol}}

\DeclareMathOperator\tr{Tr}

\DeclareMathOperator\Tr{Tr}

\def\dim{\operatorname{dim}}
\def\codim{\operatorname{codim}}
\def\crit{\operatorname{Crit}}

\def\supp{\operatorname{supp}}
\def\graph{\operatorname{graph}}

\def\Hess{\operatorname{Hess}}

\newtheorem{theorem}{Theorem}
\newtheorem{proposition}[theorem]{Proposition}
\newtheorem{corollary}[theorem]{Corollary}
\newtheorem{lemma}[theorem]{Lemma}

\theoremstyle{definition}
\newtheorem{remark}[theorem]{Remark}
\newtheorem{definition}[theorem]{Definition}

\title[Analytic and algebraic indices of elliptic operators]{Analytic and algebraic indices of elliptic operators associated with discrete groups of quantized canonical transformations}

\author{Anton Savin, Elmar Schrohe}

\date{}
\begin{document}
\begin{abstract} 
We consider elliptic operators associated with discrete groups of quantized canonical transformations. 
In order to be able to apply results from algebraic index theory, 
we define the localized algebraic index of the complete symbol of an elliptic operator.
With the help of a calculus of semiclassical quantized canonical transformations, a version of Egorov's theorem and a theorem on trace asymptotics for 
semiclassical Fourier integral operators we show that the localized analytic index and the localized algebraic index coincide.
As a corollary, we express the Fredholm index in terms of  the algebraic index for a wide class of groups, in particular, for finite extensions of Abelian groups.

\end{abstract}

\maketitle

\tableofcontents

\section{Introduction}

Given a representation of a group $G$ in a space of functions on a manifold $M$, there is an associated class of operators, called $G$-operators in the sequel,
generated by the operators in the representation and the pseudodifferential operators on $M$. This framework is quite broad and encompasses many interesting 
classes of operators. 
The case of shift operators, where the action of $G$ is induced by diffeomorphisms on $M$, 
is very well studied and goes back to the paper \cite{Carl1}  by  Carleman, where  he considers elliptic boundary value problems with nonlocal boundary conditions
associated with an involution of the boundary and reduces the problem to treating 
$G$-operators on the boundary. 
Later, the ellipticity and Fredholm property for  $G$-operators associated with general 
groups $G$ were studied by Antonevich and his coauthors, 
see \cite{AnLe1} and the references cited there.
Alain Connes' noncommutative geometry provided interesting classes of $G$-operators and efficient methods for studying them, 
see e.g. \cite{Con4,CoDu,CoLa,CoMo,DaSi,LaSu,Mos,Per1}). 
Index formulas were obtained for the actions of finite groups \cite{Ant4}, 
for isometric actions \cite{NaSaSt17,SSS1}, and for general nonisometric actions 
\cite{SaSt1,Per1}. Finally, we mention that $G$-operators have also been studied on manifolds with additional structures,
for instance,  on contact manifolds \cite{PeRo1}   and on manifolds with singularities   \cite{SaSt42}. 

In a previous joint work with Boris Sternin \cite{SSS2,SSS3} we introduced a class of 
operators associated with a representation of   $G$ by {\em quantized canonical transformations}. 
The main geometric difference between this situation and the case of shift operators is that 
here the operators are associated with an action of $G$ on $T_0^*M$, the cotangent bundle  
without the zero section, by homogeneous canonical transformations, whereas in the case of shift operators we had a group action on the base manifold.
This framework includes the theory for shift operators as a special case, but it also features 
new interesting examples: metaplectic operators,  operators associated with the wave group on 
Riemannian manifolds or boundary value problems for hyperbolic equations with conditions 
on the entire boundary, cf. \cite{BaSt1}.

In \cite{SSS2,SSS3}, we defined symbols for $G$-operators associated with  
finite-dimensional Lie groups and showed the Fredholm 
property for elliptic operators. 
For discrete groups, ellipticity requires the invertibility of the principal symbol in the 
maximal $C^*$-crossed product $C(S^*M)\rtimes G$.  
There remains the problem of computing the index of such 
operators. This is a rather difficult task; for instance, it 
contains as a special case the Atiyah--Weinstein problem of determining the 
index of quantized canonical transformations, which was solved
in \cite{EpMe1,LNT1}. 

We intend to tackle the index problem for operators associated with groups of quantized canonical transformations using the approach of algebraic 
index theory, see e.g. \cite{Fds4,Fds7,Fds16,NeTs1,NeTs2,NeTs3,PPT1,GKN1} 
and the references cited there. 
In this framework the problem of computing the Fredholm index is first reduced
to a problem of computing the so-called {\em algebraic index}, defined in terms of the complete symbols of the operators 
(in some sense this is a passage from the index in analysis to the index in algebra) and then
computing the algebraic index in terms of the principal symbol of the operator and the characteristic classes of the manifold (passage from the 
index in algebra to the index in topology). 

In this article we present a realization of the above described passage from the Fredholm index to the algebraic index. 
Let us mention the main problems we met.
First, unlike earlier applications of algebraic index theory to pseudodifferential operators, where
complete symbols are always considered on the entire cotangent bundle $T^*M$, when dealing with quantized canonical transformations, 
we have to exclude the zero section, since homogeneous canonical transformations in general have singularities at the zero section.
In the case of a single canonical transformation, Leichtnam, Nest and Tsygan 
\cite{LNT1} handled the problem by a gluing construction. With several transformations, 
this is no longer possible.  

Second, to formulate the problem in terms of complete semiclassical symbols, we also
need to convert quantized canonical transformations in our representation of the group $G$ to semiclassical quantized canonical transformations, 
while in some sense preserving
the main properties of the representation. 
We show that this is possible at least at the level of complete symbols. 
Let us remark that although there is a vast literature on semiclassical methods, see e.g.\cite{BlFo1,GuSt4,LS1,Mart1,MSS1,NOSS1,NScS12,Rob1,Zwo1},
we could not find there the  composition formulas or Egorov's theorem as needed in our situation.

We now describe the contents of the paper in more detail.
Sections \ref{sec2}--\ref{sec5}  are devoted  to the results in semiclassical analysis needed for applications  to $G$-operators. 
They might also be of independent interest.  We start Section~\ref{sec2} by recalling the main properties of semiclassical pseudodifferential operators.
Then we introduce quantized canonical transformations, define their semiclassical deformations and establish the main properties of such deformations 
(estimates in Sobolev spaces, invariance with respect to the choice of the phase function, composition formula). 
A semiclassical version of Egorov's theorem \cite{Ego1} with improved remainder estimates is obtained in Section~\ref{sec4}. In the next section, we obtain asymptotic expansions as $h\to 0$ of the operator traces of semiclassical quantized canonical transformations  under certain transversality conditions  (Theorem~\ref{lem3}). 
It yields as special cases  the expansions in \cite{StSh78,GuSt4}. In Section~\ref{sec6},  we apply the results of the previous sections to elliptic $G$-operators, where $G$ is a discrete group of quantized canonical transformations. 
To this end, we recall the definition of ellipticity for such operators and define, 
for each element $g\in G$, an analytic index localized at $g$. Then we use the results in Section~\ref{sec2} to define a semiclassical deformation of such operators.  Egorov's theorem shows that complete symbols of such operators are elements of the crossed product of the algebra  of semiclassical symbols and the group $G$. This enables us   to define traces localized at conjugacy classes in $G$ on the above mentioned  crossed product of semiclassical symbols in terms of the asymptotic expansion of the operator trace of the corresponding semiclassical operators. 
Given the crossed product algebra of  symbols and a trace on it, we define the algebraic index for elliptic (invertible) 
elements in the crossed product algebra (Definition~\ref{algind1}). Finally, we prove that the analytic index of an elliptic operator 
is equal to the algebraic index of its symbol (Theorem~\ref{th-1}), which is the main result of this paper. 
As a corollary, this result gives a formula for the Fredholm index as a sum of localized algebraic indices. 
Let us make two final remarks. First, we mention that we work throughout the paper with algebraic crossed products with $G$, 
since the main result  of this paper is mainly algebraic.
Second, using similar techniques, one can study  more general  analytic and algebraic indices associated
with cyclic cocycles over the group algebra of $G$ (cf.~\cite{Per1}).

The authors are grateful to M.~Doll, A.~Gorokhovsky, V.~Nazaikinskii, R.~Nest,  T.~Schick, 
and R.~Schulz  for useful discussions.
This work was partially supported by Deutsche Forschungsgemeinschaft, grant SCHR 319/8-1, RFBR, grant 16-01-00373a, and RUDN University  program 5-100.

\section{Semiclassical Quantized Canonical Transformations}\label{sec2}
\paragraph{\bf Semiclassical symbols.}

Let us first recall some facts regarding semiclassical symbols and operators (see \cite{MaFe1,BlFo1,NeTs3}); see also the books \cite{Mart1,GuSt4,Rob1,NScS12}.

Let $M$ be an $n$-dimensional 
closed manifold. 
We write $T^*_0M= T^*M\setminus \{0\}$ for the cotangent bundle of $M$ with the zero section removed; 
it carries the natural symplectic form $\omega = \sum_{j=1}^n dx^j\wedge d\xi_j$
for the standard coordinates $(x,\xi)$ of $T^*M$.
 
A  {\em semiclassical symbol} $a=a(x,\xi,h)$  of order $m$ in  a chart in $T^*M$ with coordinates $(x,\xi)$
is a smooth family of symbols  
with parameter $h\in  {\mathbb{R}}_+$, which has an asymptotic expansion 
\begin{equation}\label{eq-2}
a(x,\xi,h)\sim \sum_{j\ge 0} h^j a_j(x,\xi), \quad \text{as $h\to 0$},
\end{equation}
with $a_j(x,\xi)\in S^{m-j}$ in the sense  that, for all $N\ge 0$, we have
$$
 h^{-N}\Bigl(a(x,\xi,h)- \sum_{0\le j< N} h^j a_j(x,\xi)\Bigr) \longrightarrow 0 \quad \text{in }S^{m-N}\text{ as }h\to 0.
$$
The term $a_0\in S^{m}$ is called the {\em leading symbol} of $a$.

We shall identify semiclassical symbols with the same expansion \eqref{eq-2}.

Given two semiclassical symbols $a$ and $b$, their star product is defined by the formula 
\begin{equation}\label{eq-star1}
 a*b\sim\sum_{k,l,|\alpha|\ge 0} \frac{h^{k+l+|\alpha|}}{\alpha !} \partial^\alpha_x a_k(x,\xi) D^\alpha_\xi b_l(x,\xi),
 \quad \text{ where } D^\alpha_\xi=(-i)^{|\alpha|}\partial^\alpha_\xi.
\end{equation} 
This product is associative.

To define the semiclassical symbols on a manifold, we need  to describe how semiclassical symbols transform
under a change of coordinates.
Given another coordinate system $(y,\eta)$ and a change of coordinates 
$$
(y,\eta)=\partial f (x,\xi) \qquad \text{or}\qquad y=f(x),  \eta=\left(\left(\frac{ \partial f}{\partial x}\right)^t\right)^{-1}\xi,
$$
the semiclassical symbol \eqref{eq-2} in the new coordinate system is written as 
\begin{equation}\label{eq-semicl1}
 b(y,\eta,h)=(\partial f^{-1})^* (\mathcal{L} a(x,\xi,h)),
\end{equation}
where 
$$
\mathcal{L} a(x,\xi,h)=a+\sum_{|\alpha|\ge 2}\frac{(-ih)^{|\alpha|}}{\alpha!}
  \left(\frac{\partial^\alpha}{\partial t^\alpha}\exp\frac{i\langle\xi, (f'_x)^{-1}(f(x+t)-f(x)-f'_x t)\rangle}{h}
  \right)\Bigr|_{t=0}\frac{\partial^\alpha a}{\partial \xi^\alpha}.
$$
A {\em semiclassical symbol} on  $M$ is a collection of semiclassical symbols in local charts such that the  symbol 
$a$ in the chart with coordinates $(x,\xi)$ and the symbol $b$ in the chart with coordinates $(y,\eta)$ are related as in
Eq.~\eqref{eq-semicl1} on the intersection of the two charts.
Denote the space of semiclassical symbols of orders $\le 0$ on   $M$ by $\mathbb{A}$. The product \eqref{eq-star1} is compatible with
the change of variables formula \eqref{eq-semicl1} and therefore $\mathbb{A}$ is an associative algebra with respect to the star product.
This algebra is filtered by the ideals $\mathbb{A}_N\subset \mathbb{A}$, which consist of semiclassical symbols \eqref{eq-2}
such that the summation starts with $j=N$.

Let  $\mathbb{A}'\subset \mathbb{A}$ be the ideal of semiclassical symbols which vanish in a neighborhood of the zero section in $T^*M$. More precisely, a symbol
$a\in \mathbb{A}$  is in $\mathbb{A}'$, if, for each $N\ge 1$, 
there exists a neighborhood $U$  of the zero section such that the first $N$ 
components of $a$ in the expansion \eqref{eq-2} are equal to zero in $U$.
We again have a filtration by the ideals $\mathbb{A}'_N\subset \mathbb{A}'$, where the summation starts with $j=N$.

\paragraph{\bf Semiclassical pseudodifferential operators.}
The definition of the quantization mapping  $Op_h$, which  takes semiclassical  symbols to pseudodifferential operators, is the same
as in the classical theory of pseudodifferential operators. 
Choose a partition of unity $\sum_k \chi_k(x)\equiv 1$ subordinate to an atlas of charts on $M$ and cut-off functions $\psi_k(x)$
supported in these charts such that $\chi_k\psi_k=\chi_k$.
Given a semiclassical symbol $a\in\mathbb{A}$, the corresponding  semiclassical pseudodifferential operator is given by
\begin{equation}\label{eq-q1}
 Op_h(a)=\sum_k \psi_k A_k \chi_k, 
\end{equation}
where $A_k$  is the semiclassical quantization in local coordinates defined by the oscillatory integral
$$
 A_ku(x)=\frac{1}{(2\pi h)^n} \iint e^{\frac{i}{h}(x-y)\xi}a(x,\xi,h)u(y)dyd\xi.
$$
We identify two semiclassical pseudodifferential operators if their difference is  of order $-\infty$ in the scale of Sobolev spaces for all $h\in (0,1]$ and  its norm in any pair of Sobolev spaces tends to zero  faster than any power of $h$ as $h\to 0$. Denote by   $\Psi_h(M)$ the space of {\em semiclassical pseudodifferential operators} defined as the quotient of the range of this quantization mapping in the space of operator families with parameter $h\in (0,1]$ acting on $C^\infty(M)$  under the above identification. 

It can be shown that semiclassical pseudodifferential operators form an algebra under composition of operator families and the
quantization mapping induces an isomorphism
\begin{equation}\label{eq-iso1}
 Op_h:\mathbb{A}\longrightarrow \Psi_h(M)
\end{equation}
between the spaces of semiclassical symbols and  operators. This isomorphism is independent of the choice of  atlas, 
partition of unity and cut-off functions.

\paragraph{\bf Quantized Canonical Transformations.}

A homogeneous canonical transformation is a symplectomorphism 
$$C:T^*_0M\longrightarrow T^*_0M,
$$
which is homogeneous of degree one in the fiber. 
Alternatively, we may ask that $C$ preserves the canonical one-form $ \sum_{j=1}^n \xi_j dx^j$. 

The graph of $C$, i.e. the set 
$$\text{\rm graph}\, C = \{(x,\xi; x',\xi')\in T^*_0M\times T^*_0M\mid (x,\xi) = C(x',\xi')\}$$
is a Lagrangian submanifold of $ T^*_0M\times T^*_0M$.
Given a point $(x_0,\xi_0; x'_0,\xi'_0)\in \text{\rm graph}\,C$  we find neighborhoods 
$U_{x_0}$ of $x_0$ and $U_{x'_0}$ of $x'_0$ as well as an open conic subset 
$\Gamma\subset \mathbb R^d$ 
and a nondegenerate phase function%
\footnote{ i.e., $\varphi$ is real-valued, homogeneous of degree 1 in $\theta$,
$d_{(x,x',\theta)}\varphi\ne 0$ and the
differentials $d_{(x,x',\theta)}(\partial_{\theta_j}\varphi)$, $j=1,\ldots,d$,  are linearly independent whenever 
$\partial_\theta \varphi=0$. }
$\varphi: U_{x_0}\times U_{x'_0} \times \Gamma\longrightarrow \mathbb{R}$ such that the map
\begin{eqnarray}\label{eq-a1} 
\alpha: \text{\rm Crit}_\varphi \to \text{\rm graph}\, C; \quad 
\alpha (x,x',\theta) = (x,\partial_x\varphi(x,x',\theta), x', -\partial_{x'}\varphi(x,x',\theta))
\end{eqnarray}
is a diffeomorphism from the set 
\begin{equation}\label{eq-crit1}
 \crit_\varphi=\{(x,x',\theta)\;|\; \partial_\theta\varphi(x,x',\theta)=0\}
\end{equation} 
of critical points of $\varphi$ to a conical neighborhood of $(x_0,\xi_0;x'_0,\xi'_0)$ on $\text{\rm graph}\, C$.

An operator 
$$\Phi:C^\infty(M)\to C^\infty(M)
$$ 
is called a {\em quantized canonical transformation}, if 
microlocally,  in conical neighborhoods of the points $(x_0,\xi_0)$ and $(x'_0,\xi'_0) = C^{-1}(x_0,\xi_0)$, 
the operator $\Phi$  has an integral kernel that can be written as an oscillatory integral
\begin{equation}\label{ker1}
K^\Phi(x,x')=\int e^{i\varphi(x,x',\theta)} b(x,x',\theta) d\theta,
\end{equation}
modulo smooth kernels, where 
\begin{itemize}
\item $\varphi(x,x',\theta)$ is a nondegenerate phase function as described above.
\item  $b\in S^{(n-d)/2}$ is an amplitude function, which vanishes outside a small conical  neighborhood of $\crit_\varphi$ as well as  for small $|\theta|$ and has an asymptotic expansion 
$$
b(x,x',\theta)\sim \sum_{j\ge 0} b_j(x,x',\theta)  \text{ as $|\theta|\to \infty$},
$$
with $b_j(x,x',\theta)$ homogeneous of degree $(n-d)/2-j$ in $\theta$. 
\end{itemize}  

\begin{remark}\label{regularization}
The phase function $\varphi$ is assumed to be homogeneous of degree $1$ in
$\theta$ and therefore might be singular near $\theta=0$. However, since $b$ 
vanishes near $\theta=0$ we may as well assume $\varphi$ to be a classical symbol in 
$S^1$.    
\end{remark}

\paragraph{\bf Semiclassical quantized canonical transformations.}

Given  a quantized canonical transformation $\Phi$ with its Schwartz kernel microlocally written in the form \eqref{ker1}
and  $\varepsilon >0$, $N\in \mathbb{N}$, we define the semiclassical Fourier integral operator $\Phi_{h,\varepsilon,N}$ associated with $\Phi$ as the operator with the integral  kernel
\begin{equation}\label{ker2}
K^\Phi_{h,\varepsilon,N}(x,x') = h^{-d/2-n/2}\int e^{\frac i h \varphi(x,x',\theta)} \sum_{0\le j< N}h^j b_j(x,x',\theta) \chi (x,x',\theta) d\theta,
\end{equation}
where the smooth function $\chi$ is chosen such that, with 
$\alpha$ defined in \eqref{eq-a1},
\begin{eqnarray}\label{chi}
\mbox{\ \ \ \ \ \ \ } \chi(x,x',\theta)=
 \left\{
  \begin{array}{ll}
   1 & \text{in an open neighborhood of  }\alpha^{-1}(T^*_0M\times\{|\xi'|\ge\varepsilon\}),\\ 
  & \text{which is conic at infinity,} \vspace{2mm}\\
   0 & \text{in a small neighborhood of the zero section}.   
  \end{array}
 \right.
\end{eqnarray}
From now on we fix some Riemannian metric on $M$ which allows us to define $|\xi'|$. We call a subset $U$ of $M\times M\times \mathbb R^d$ conic at infinity, 
if  there exists an $R>0$ such that $(x,x',\lambda \theta)\in U $ whenever $(x,x',\theta) \in U$, $|\theta|\ge R$,  and $\lambda\ge 1$. 
By possibly shrinking the domain slightly, we can (and will) assume that $\chi\equiv 1$ for large $|\theta|$.

It follows from the homogeneity of the phase function and the $b_j$ that 
\begin{eqnarray*}
K^\Phi_{h,\varepsilon,N}(x,x') 
&=&h^{(d-n)/2}\int e^{i \varphi(x,x',\theta)} \sum_{0\le j< N}h^j b_j(x,x',h\theta) 
\chi (x,x',h\theta) d\theta\\
&=&\int e^{i \varphi(x,x',\theta)} \sum_{0\le j< N} b_j(x,x',\theta) 
\chi (x,x',h\theta) d\theta.
\end{eqnarray*}

\paragraph{\bf Boundedness results.}

We recall the semiclassical Sobolev spaces:
\begin{definition}
\label{Hsh}
The space $H_h^s(M)$ is the set of all distributions $u$ on $M$
such that
$$\|u\|_{H^s_h}= \|(h^2\Delta +1)^{s/2}u\|_{L^2}<\infty,$$ 
where $\Delta$ stands for the nonnegative Laplacian on $M$.
\end{definition}

\begin{lemma}\label{mapping}
Let $\varphi$ and $\chi$ be as before, and let $p(h;\cdot,\cdot, \cdot)$, $0<h\le 1$,  
be a bounded family of symbols in  $ S^{(n-d)/2-N}$ for some  $N\in \mathbb N_0$.  
Denote by $\Psi_h$ the operator given  by the Schwartz kernel  
\begin{eqnarray*}
K^\Psi_h (x,x') =  \int e^{ i\varphi(x,x',\theta)}
p(h;x, x',\theta) \chi(x,x',h\theta)\, d\theta.
\end{eqnarray*}
Then 
$$\Psi_h: H^s_h(M) \to H^{s+N}_h(M)$$
is bounded for every $s\in \mathbb R$, and
\begin{eqnarray}\label{h-estimate}
\|\Psi_h\|_{\mathcal B(H^{s}_h,H^{s+N}_h)} = O(h^N).
\end{eqnarray}
\end{lemma} 
We call an operator family satisfying \eqref{h-estimate} an $O(h^N)$-family.

\begin{proof}
By interpolation and duality, it is sufficient to establish the assertion for $s\in \mathbb N_0$. Choose an elliptic symbol $a=a(y,\eta)\in S^{-s}$, which vanishes for small $\eta$, 
and consider the composition $\Psi_h Op_h(a)$. 
It is then sufficient to show that for every multi-index $\alpha$ with $|\alpha|\le s+ N$, 
the norm of the composition $(hD_x)^\alpha \Psi_hOp_h(a)$ in $\mathcal B(L^2(M))$ is $O(h^N)$. 

In view of the fact that 
\begin{eqnarray*}\lefteqn{D_{x_j} \Psi_h(x,x') 
= \int e^{i\varphi(x,x',\theta)}}\\
&&\times \Big(\partial_{x_j} \varphi(x,x',\theta) p(h;x, x',\theta)\chi(x,x',h\theta)
+ D_{x_j} \big(p(h;x,x',\theta) \chi(x,x',h\theta)\big)\Big) \, d\theta
\end{eqnarray*}
and that derivatives of $\chi(x,x',h\theta)$ are in $S^0$, uniformly in $h$  we may reduce the task to the following: 
Let $ \tilde p_k(h;\cdot,\cdot,\cdot)$, $0<h\le 1$, be bounded in $S^{(n-d)/2-N+k}$ for $k=0,\ldots,s+N$, denote by 
$\Psi_h^{(k)}$ the operator defined by the Schwartz kernel 
$$\int e^{i\varphi(x,x',\theta)} h^k \tilde p_k(h;x,x',\theta) \chi(x,x',h\theta)\, d\theta$$ 
and show that $\|\Psi^{(k)}_hOp_h(a)\|_{\mathcal B(L^2(M))}=O(h^N)$.
Indeed, the composition  $\Psi^{(k)}_hOp_h(a)$ has the Schwartz kernel 
$$(2\pi)^{-n}  \int\!\!\!\iint e^{ i\psi(x,x',y,\theta,\eta)} c(h;x,y,\theta,\eta) \,  d\eta dy d\theta,$$
where $\psi(x,x',y,\theta,\eta) = \varphi(x,y,\theta) +(y-x')\cdot\eta$ and 
$$
c(h; x,y,\theta, \eta) = h^k \tilde p_k(h;x,y,\theta) \chi(x,y,h\theta) a(y,h\eta). 
$$
As the phase function  $\varphi$ parametrizes a neighborhood of a point on the graph of $C$, we can assume 
that $\partial_y \varphi(x,y,\theta)\not=0$ for all $(x,y,\theta)$. 
Following an idea from the proof of \cite[Theorem 25.2.3]{HIV} we show that there exist constants $c_1,c_2>0$
such that 
$$c_1|\eta| \le |\theta|\le c_2|\eta|, \ \text{ if } \partial_y\varphi(x,y,\theta) + \eta=0:$$
In fact, if $|\eta|=1$, then the $1$-homogeneity implies that $|\theta|$ is bounded from above and below, when $|\partial_y \varphi(x,y,\theta)| =1=|\eta|$. We next choose a smooth function $H=H(\theta, \eta)$, homogeneous of degree zero,  which is equal to 1, if 
$ c_1|\eta|/2 \le |\theta|\le 2c_2|\eta|$ and has its support in the conic set, where
$c_1|\eta|/3 \le |\theta|\le 3c_2|\eta|$.  We  define 
\begin{eqnarray*}
d(h;x,y,\theta,\eta) &=& c(h; x,y,\theta,\eta) H(\theta,\eta), \text{ and}\\
r(h;x,y,\theta,\eta) &=& c(h; x,y,\theta,\eta)(1- H(\theta,\eta)).
\end{eqnarray*}
Then there exists a constant $c_0$ such that 
$|\partial_y\varphi(x,y,\theta)+\eta|\ge c_0(|\theta|+|\eta|) $ on $\text{\rm supp }r$ and 
$c_1|\eta|/3 \le |\theta|\le 3c_2|\eta|$ on $\text{\rm supp }d$. 

Let us first consider 
$$R(x,x') = \int\!\!\iint e^{i \psi(x,x',y,\theta,\eta)} r(h;x,y,\theta,\eta) \, d\theta d\eta dy.
$$
We define the first order differential operator
\begin{eqnarray*}
L = L(x,y,\theta,\eta,D_y) 
= \frac{\partial_y\varphi(x,y,\theta) +\eta}{|\partial_y\varphi(x,y,\theta) +\eta|^2}D_y
\end{eqnarray*}
and observe that $Le^{i \psi(x,x',y,\theta,\eta)} = e^{i \psi(x,x',y,\theta,\eta)}$. 
For any $j\in \mathbb N$ we therefore have  
\begin{eqnarray}\label{eq-rr}
 R(x,x') = \int\!\!\iint e^{ i \psi(x,x',y,\theta,\eta)} (L^t)^j r(h;x,y,\theta,\eta) \, d\theta d\eta dy.
\end{eqnarray}
In view of the fact that there exists some $\delta >0$ such that $\chi(x,y,\theta)$ and 
$a(y,\eta)$ vanish for $|\theta|\le \delta$ and $|\eta|\le \delta$, respectively, we see that the 
amplitude vanishes unless $h( |\theta|+|\eta|)\ge 2\delta$. In this case, however, there 
exists a constant $c_\delta$ such that 
\begin{eqnarray}\label{est.1}
(1+|\theta|+|\eta|)^{-1} \le c_\delta h (1+ h|\theta| + h|\eta|)^{-1}.
\end{eqnarray}
Hence the amplitude in \eqref{eq-rr} as well as its derivatives can be estimated by 
$O(h^\ell (1+|\theta|+|\eta|)^{-\ell})$ 
for arbitrary $\ell$, if we choose $j$ sufficiently large.  
So we obtain the assertion in this case. 

Next consider the operator with the kernel  
$$D^h(x,x') = \int\!\!\iint e^{ i \psi(x,x',y,\theta,\eta)} d(h;x,y,\theta,\eta) \, d\theta d\eta dy.
$$ 
Since $\tilde p_k(h;\cdot, \cdot, \cdot)\in S^{(n-d)/2-N+k}$,  
uniformly in $h$, and $a\in S^{-s}$  we conclude that, with suitable constants  
$c_{\alpha, \beta,\gamma}$,
$$
|D^\alpha_\theta D^\beta_{x,y} D^\gamma_\eta d(h;x,y,\theta, \eta)|
\le \begin{cases} c_{\alpha, \beta,\gamma} h^k\langle\theta\rangle^{(n-d)/2-N+k-|\alpha|} h^{|\gamma|} \langle h\eta\rangle^{-s-|\gamma|} , &  h|\theta|\ge \delta\\0&\text{else}.
\end{cases}
$$
We note that $h^k\langle \theta\rangle^{k} \le \langle h\theta\rangle^k$, 
that $\langle \theta\rangle^{-N} \le c'_\delta h^N \langle h\theta\rangle^{-N}$ for $|h\theta|\ge \delta$ and some $c_\delta'$, and that $h^{|\gamma|}\langle h\eta\rangle^{-|\gamma|} \le \langle\eta\rangle^{-|\gamma|}$.
Since, moreover,  $c_1|\eta|/3 \le |\theta|\le 3c_2|\eta|$ on $\text{\rm supp }d$ we can estimate 
the last expression by $c_3 h^N(1+|\theta|+|\eta|)^{(n-d)/2-|\alpha|-|\gamma|}$, 
uniformly in $h$, for a suitable constant $c_3$. 
Hence $h^{-N} D^h$ is a conormal distribution with uniform bounds in 
$I^0(M\times M,C)$, and the associated operator is
bounded on $L^2(M)$, uniformly in $h$, by \cite[Corollary 25.3.2]{HIII}. 
This completes the argument. 
\end{proof}

\begin{lemma}\label{negligible}
Consider a semiclassical Fourier integral operator $\Psi$ given by a kernel of the form 
\begin{eqnarray*}%\lefteqn{}\\
K^\Psi (x,x') = 
\int e^{\frac ih\varphi(x,x',\theta)} p(h;x,x',\theta)\, d\theta,
\end{eqnarray*}
where the phase function $\varphi$ is as before and $p(h; \cdot, \cdot, \cdot)$, $0<h\le 1$, 
is a uniformly bounded family in $S^{(n-d)/2}$. We assume that $p(h; x,x',\theta)$  
vanishes for small and large $|\theta|$, uniformly in $h$,  and, moreover, in an open
neighborhood of $\alpha^{-1}(T^*_0M\times \{|\xi'|\ge \varepsilon\})$, which is conic at infinity,
also independent of $h$.  
Furthermore, let $a=a(y,\eta)$ be a symbol which vanishes for $|\eta|\le  2\varepsilon$. 
Then $\Psi Op_h(a)$ is an $O(h^\infty)$-family. 
\end{lemma} 

\begin{proof} The Schwartz kernel of $\Psi Op_h(a)$ is  given by  
the oscillatory integral
\begin{eqnarray*}\lefteqn{
(2\pi)^{-n}  h^{ -n} \int\!\!\!\iint e^{\frac ih \psi(x,x',y,\theta,\eta)} p(h;x,y,\theta)a(y,\eta) \,  d\eta dy d\theta}\\
&=& (2\pi)^{-n} h^d \int\!\!\!\iint e^{i \psi(x,x',y,\theta,\eta)} c(h;x,y,\theta,\eta) \,  d\eta dy d\theta,
\end{eqnarray*}
where $\psi(x,x',y,\theta,\eta) = \varphi(x,y,\theta) +(y-x')\cdot\eta$ and 
$$
c(h; x,y,\theta, \eta) = p(h;x,y,h\theta) a(y,h\eta). 
$$
We are now in a situation similar to that in the proof of Lemma \ref{mapping} and let
\begin{eqnarray*}
d(h;x,y,\theta,\eta) &=& c(h; x,y,\theta,\eta) H(\theta,\eta), \text{ and}\\
r(h;x,y,\theta,\eta) &=& c(h; x,y,\theta,\eta)(1- H(\theta,\eta)).
\end{eqnarray*}
with $H$ as before, recalling that  there exist positive constants $c_0$, $c_1$, $c_2$, 
such that 
$|\partial_y\varphi(x,y,\theta)+\eta|\ge c_0(|\theta|+|\eta|) $ on $\text{\rm supp }r$ and 
$c_1|\eta|/3 \le |\theta|\le 3c_2|\eta|$ on $\text{\rm supp }d$. 
As in the proof of Lemma \ref{mapping} we see that the Schwartz kernel 
$$R^h(x,x') = \int\!\!\iint e^{ i \psi(x,x',y,\theta,\eta)} r(h;x,y,\theta,\eta) \, d\theta d\eta dy,
$$
defines an $O(h^\infty)$-family. 

We shall show that the same is true for the operator with the kernel  
\begin{eqnarray}\label{eq-dh1}
 D^h(x,x') = 
\int\!\!\iint e^{\frac ih \psi(x,x',y,\theta,\eta)} d\left(h;x,y,\frac \theta h,\frac \eta h\right) \, d\theta d\eta dy.
\end{eqnarray}
Here, the amplitude $d$ can be assumed to have compact support, since 
$p$ vanishes for large $|\theta|$ and since $|\eta|$ and $|\theta|$ are comparable on the  support of $d$. 
   
The phase $\psi$ is stationary, if and only if 
$$\partial_\theta\varphi(x,y,\theta)=0, \quad y=x', \quad \eta=-\partial_y\varphi(x,y,\theta) .
$$
In this case,  $(x,x',\theta)$ belongs to the critical set of $\varphi$ and $\eta$ equals the 
$\xi'$-component of $\alpha(x,x',\theta)$. By assumption, however,  
$p$ vanishes in an open 
neighborhood of $\alpha^{-1}(T^*_0M\times \{ |\xi'|\ge \varepsilon\})$, while $a(y,\eta)$ 
vanishes for $|\eta|\le 2\varepsilon$.  
As a result, the gradient of the phase function is bounded away  from zero on the 
support of the amplitude. 
Integration by parts in \eqref{eq-dh1} with the operator 
$$\tilde L=
\frac{\partial_\theta\varphi(x,y,\theta)\partial_\theta+ (y-x')\partial_\eta + (\eta+\partial_y\varphi(x,y,\theta))\partial_y}{|\partial_\theta\varphi(x,y,\theta)|^2+|y-x'|^2+|\eta+\partial_y\varphi(x,y,\theta)|^2}
$$
in connection with Estimate \eqref{est.1}
then shows that the integrand together with its derivatives is 
$O(h^\ell (1+|\theta|+|\eta|)^{-\ell})$ for arbitrary $\ell$ in $\mathbb N$.  
Hence \eqref{eq-dh1} defines an $O(h^\infty)$-family.  
\end{proof}

\paragraph{\bf Correctness of the definition.}

Consider the Schwartz kernels \eqref{ker2} for two choices 
$\chi$ and $\tilde \chi$ of excision functions with the properties in \eqref{chi}. 
Denote, for the moment, the associated kernels by $K^\Phi_{h,\varepsilon, N}$ and
$\tilde K^\Phi_{h,\varepsilon, N}$. 
From the fact that $\chi-\tilde \chi$ vanishes in  an open neighborhood of 
$\alpha^{-1} (T^*_0M\times \{|\xi'|\ge \varepsilon\})$, 
which is conic near infinity, and  Lemma \ref{negligible}
we conclude: 

\begin{corollary}\label{different_chi}
If $a$ is a pseudodifferential symbol with $a(y,\eta)=0$ for $|\eta|\le 2\varepsilon$ 
and $\Phi_{h,\varepsilon,N}$ and $\tilde \Phi_{h,\varepsilon,N}$ are the operators associated with the Schwartz kernels $K^\Phi_{h,\varepsilon, N}$ and $\tilde K^\Phi_{h,\varepsilon, N}$, then 
\begin{eqnarray}\label{Diff}
( \Phi_{h,\varepsilon,N}-\tilde \Phi_{h,\varepsilon,N})Op_h(a) \text{ is an }O(h^\infty)\text{-family.}
\end{eqnarray}
\end{corollary}

\paragraph{\bf Changing phase and amplitude}
The kernel $K^\Phi$ in \eqref{ker1} can be represented by various choices of a phase 
function $\varphi$ and an amplitude $b$. 
To what extent does the kernel $K^\Phi_{h,\varepsilon,N}$ in 
\eqref{ker2} depend on these choices?
In Proposition \ref{fio1}, below, we will show the following result: 
The operator family $\Phi_{h,\varepsilon,N}$ is independent 
of the choice modulo $O(h^{N})$-operator families and families whose Schwartz kernel has the properties in Lemma 
\ref{negligible}.
 
It is known from the work of H\"ormander \cite[Section 3]{Hoe1}  that the changes in phase and amplitude can be reduced to two cases, namely those, 
where the new phase and the new amplitude arise from the previous ones by either 
a change of coordinates or an increase, respectively  a reduction of $\theta$-variables. 
We consider the two cases separately.

\paragraph{\em Change of variables.}  

Let $U$ be a conic open set in $\mathbb R^n\times \mathbb R^n\times \mathbb R^d$ 
which maps to a neighborhood of the point $(x_0,\xi_0; C^{-1}(x_0,\xi_0))$
under the diffeomorphism $\alpha$ in \eqref{eq-a1}.
Suppose there exist another  conic open set $\tilde U\subseteq \mathbb R^n\times \mathbb R^n\times \mathbb R^d$ and a map 
$s:\tilde U\to \mathbb R^d$, positively homogeneous of degree one in the $\mathbb R^d$-variable,  such that the map $S: \tilde U \to U$ given by 
$S(x,x',\eta) = (x,x',s(x,x',\eta))$
is a diffeomorphism of open sets. 

This leads to another representation of the kernel $K^\Phi$ in \eqref{ker1}: Define $\tilde \varphi:\tilde U\to \mathbb R$, $\tilde b:\tilde U\to \mathbb C$ by 
\begin{eqnarray*}
\tilde \varphi(x,x',\eta) &=& \varphi(x,x',s(x,x',\eta))\\
\tilde b(x,x',\eta) &=& b(x,x',s(x,x',\eta)) J_s(x,x',\eta),
\end{eqnarray*}  
where 
$J_s$ is the Jacobian, defined by 
$J_s(x,x',\eta) =\big|\det \partial_\eta s (x,x',\eta)\big|$, and write 
\begin{eqnarray}\label{ker1a}
K^\Phi(x,x') = \int e^{i\widetilde\varphi(x,x',\eta)} \tilde b(x,x',\eta) \, d\eta.
\end{eqnarray}

\begin{lemma}\label{Lemma5} Let $K^\Phi_{h,\varepsilon,N}$ and 
$\tilde K^\Phi_{h,\varepsilon,N}$ be the kernels defined from the representations 
\eqref{ker1} and \eqref{ker1a}, respectively, of $K^\Phi$ in the sense of \eqref{ker2}. 
Then the difference between both kernels has the
properties in Lemma \ref{negligible}.
\end{lemma} 

\begin{proof}Note that $S$ provides a diffeomorphism between  $\text{Crit}_{\tilde\varphi}$ and 
$\text{Crit}_\varphi$, since 
$$\partial_\eta \tilde \varphi(x,x',\eta) = \partial_\theta \varphi(x,x',s(x,x',\eta)) 
\partial_\eta  s(x,x',\eta),$$
and $\partial_\eta  s(x,x',\eta)$ is invertible. 
Moreover, we observe that the asymptotic expansion of $\tilde b$ is given by 
$\tilde b\sim\sum \tilde b_j$, where $\tilde b_j(x,x',\eta) = b_j(x,x',s(x,x',\eta)) J_s(x,x',\eta)$ 
in view of the fact that $s$ is homogeneous of degree $1$ and $J_s$  is homogeneous 
of degree $0$. 

Next choose an excision function $\tilde \chi: \tilde U \to \mathbb R$ such that 
$\tilde \chi$ vanishes for small $|\eta|$ and $\tilde \chi\equiv 1$ in an open 
neighborhood of $\alpha^{-1} (T_0^*M\times \{|\xi'|\ge \varepsilon\})$,
conic at infinity.   

Write $\chi(x,x',\theta) = \tilde \chi(x,x',s^{-1}(x,x',\theta))$ 
with the inverse function $s^{-1}$ to $s$. Then
\begin{eqnarray*}
\lefteqn{\tilde K^\Phi_{h,\varepsilon,N}(x,x')  
= \int e^{i\tilde\varphi(x,x';\eta)}
\sum_{j< N}\tilde b_j (x,x',\eta)\tilde \chi(x,x',h\eta) \, d\eta }\\
&=& \int e^{i\varphi(x,x';\theta)}
\sum_{j< N} b_j (x,x',\theta) \chi(x,x',h\theta) \, d\theta.  
\end{eqnarray*}
It remains to show that $\chi(x,x'.\theta)$ is an excision function with the properties 
in \eqref{chi}.  
For $(x,x',\theta)\in \text{Crit}_{ \varphi}$ we have  
\begin{eqnarray*}\lefteqn{\partial_{x'} \varphi(x,x',\theta)}\\
&=&
 (\partial_{x'}\tilde \varphi)(x,x',s^{-1}(x,x',\theta))+ (\partial_{\theta}\tilde \varphi)(x,x',s^{-1}(x,x',\theta))\partial_{x'}s^{-1}(x,x',\theta),
\end{eqnarray*}
where the second summand vanishes, since $(x,x',s^{-1}(x,x',\theta))\in \text{Crit}_{\tilde\varphi}$. 
Therefore, if $(x,x',\theta) \in \text{\rm Crit}_{\varphi}$ and $|\partial_{x'}  \varphi(x,x',\theta)|\ge \varepsilon$, then $(x,x',s^{-1}(x,x',\theta))\in \text{Crit}_{\tilde \varphi}$ and 
$| (\partial_{x'}\tilde \varphi)(x,x',s^{-1}(x,x',\theta))|\ge \varepsilon$. Since $\tilde \chi$ 
is assumed to vanish in a neighborhood of these latter points which is conic at infinity, 
$\chi$ will vanish in a neighborhood of the former points, which is conic at infinity.  
Moreover, $\chi(x,x',\theta)$ also vanishes, if $|\theta|$ is small.
Hence $K^\Phi_{h,\varepsilon,N}$ and $\tilde K^\Phi_{h,\varepsilon,N}$ differ at most in the 
choice of the excision function. 
\end{proof}

\paragraph{\em  Reduction of $\theta$-variables.}
In this case one chooses another phase function describing $C$ with a smaller number of 
$\theta$-variables. We follow the set-up for the proof of Theorem 2.3.4 in Duistermaat \cite{Dui}.  
Assume that there exists  a conic neighborhood of a point $(x_0,x'_0,\theta_0)$ and,
after a linear transformation of coordinates, 
a splitting  $\theta = (\theta',\theta'')$ 
with $\theta'=(\theta_1,\ldots, \theta_{k})$,
$\theta''=(\theta_{k+1},\ldots, \theta_{d})$,   $0<k<d$, such that
$\partial^2_{\theta',\theta'}\varphi(x_0,x'_0,\theta_0)=0$, 
$\partial^2_{\theta',\theta''}\varphi(x_0,x'_0,\theta_0)=0$
and $\partial^2_{\theta'',\theta''}\varphi(x_0,x'_0,\theta_0)$ is non-degenerate. Without loss of generality 
we assume that $\theta_0''=0$. The implicit function theorem 
implies that there exists a function $s=s(x,x',\theta')$, defined in a conic neighborhood of 
$(x_0,x_0',\theta_0')$, such that the (unique) solution to $\partial_{\theta''}\varphi(x,x', \theta) = 0$
near $(x_0,x'_0,\theta_0)$ is given by $(x,x',\theta',s(x,x',\theta'))$ with 
$s(x_0,x'_0,\theta_0') = \theta_0''=0$. 

One defines $\tilde \varphi(x,x',\theta')= \varphi(x,x',\theta',s(x,x',\theta'))$ and writes 
$$\varphi(x,x',\theta',\theta'') = \tilde \varphi(x,x', \theta') + \psi(x,x',\theta',\theta'').$$

Then the kernel in \eqref{ker1} takes the form
\begin{eqnarray}\label{K_red}
K^\Phi (x,x') =  
\int  e^{i\tilde\varphi(x,x',\theta')}\int e^{i\psi(x,x',\theta',\theta'')}  b(x,x',\theta',\theta'')\, d\theta'' \, d\theta',
\end{eqnarray}
so that $K^\Phi$ has the alternative representation with the phase $\tilde \varphi$ and
the amplitude 
\begin{eqnarray}\label{tilde_b}
\tilde b(x,x',\theta') = \int e^{i\psi(x,x',\theta)} b(x,x',\theta) \, d\theta''. 
\end{eqnarray}
It can be shown that $\tilde b$ is a classical symbol in $S^{(n-k)/2}$; for a full proof see \cite[p.~144]{Hoe1}. We write $\tilde b \sim\sum \tilde b_j$ 
for the corresponding asymptotic expansion. Note that 
\begin{eqnarray}\label{eq.diff}
\int e^{i\psi(x,x',\theta)} \sum_{j< N} b_j(x,x',\theta) \, d\theta''-\sum_{j< N} \tilde b_j(x,x',\theta')=:r(x,x',\theta')
\end{eqnarray} 
is an element of $S^{(n-k)/2-N}$.

\begin{lemma}\label{Lemma7}
The two representations for the kernel $K^\Phi$ given by  
\eqref{ker1} and  \eqref{K_red}/\eqref{tilde_b} lead to 
kernels $K_{h,\varepsilon,N}^\Phi $ and $\tilde K_{h,\varepsilon,N}^\Phi$. The associated 
operators differ only by  $O(h^{N})$-families and operators whose Schwartz kernels 
have the properties described in Lemma \ref{negligible}. 
\end{lemma}

\begin{proof}
Choose an excision function $\chi$ with the properties in \eqref{chi}. 
We claim that also the function $\tilde \chi (x,x',\theta')= \chi (x,x',\theta',s(x,x',\theta'))$
is admissible in the sense of \eqref{chi} for the corresponding 
map $\tilde\alpha: \text{\rm Crit}_{\tilde\varphi} \to \text{\rm graph }C$ given by 
$$\tilde \alpha (x,x',\theta' ) =  
(x, \partial_x\tilde \varphi(x,x',\theta'); x',-\partial_{x'} \tilde \varphi (x,x',\theta')).
$$
By construction,  $s$ is one-homogeneous in $\theta'$, so $\tilde \chi$ vanishes for small $|\theta'|$.
Suppose that $(x,x',\theta')\in \tilde\alpha^{-1}(T_0^*M\times\{|\xi'|\ge \varepsilon\})$, 
i.e.~$(x,x',\theta')\in \text{Crit}_{\tilde\varphi}$ with $|\partial_{x'} \tilde\varphi (x,x',\theta')|\ge \varepsilon$. The definition of $s$ implies that  
\begin{eqnarray*}
\partial_{\theta'}\tilde \varphi(x,x',\theta') &=&
(\partial_{\theta'}\varphi)(x,x',\theta',s(x,x',\theta')) \text{ and}
\\
\partial_{x'}\tilde \varphi(x,x',\theta') &=& (\partial_{x'}\varphi)(x,x',\theta', s(x,x',\theta') ) 
\end{eqnarray*}
which implies that  $(x,x',\theta',s(x,x',\theta'))\in \alpha^{-1}(T^*_0M\times\{|\xi'|\ge \varepsilon\})$. 
As $\chi$ equals $ 1$ in a neighborhood of these points, $\tilde \chi \equiv 1$ near 
$(x,x',\theta')$.  This neighborhood is conic near infinity due to the $1$-homogeneity of $s$, 
which proves the claim.

The Schwartz kernel associated with the representation 
\eqref{K_red}/\eqref{tilde_b} of $K^\Phi$ and the choice of $\tilde \chi$ as an excision function 
is given by
$${\tilde K}^\Phi_{h,\varepsilon, N}(x,x')  = \int e^{i\tilde \varphi (x,x',\theta')} \sum_{j<N}
\tilde b_j(x,x',\theta') \tilde \chi(x,x',h\theta')\, d\theta';$$
it differs from $K^{\Phi}_{h,\varepsilon,N}(x,x')$ by 
\begin{eqnarray*}\lefteqn{
(K^{\Phi}_{h,\varepsilon,N}-{\tilde K}^\Phi_{h,\varepsilon, N})(x,x') 
= \int e^{i\tilde\varphi(x,x',\theta')}r(x,x',\theta') \tilde\chi(x,x',h\theta') \, d\theta' }\\
&&+
\int e^{i\varphi(x,x',\theta)} \sum_{j<N} b_j(x,x',\theta) (\chi(x,x',h\theta)- \tilde \chi(x,x',h\theta') )\, d\theta.   
\end{eqnarray*}
The first integral on the right hand side furnishes an $O(h^N)$-family by Lemma 
\ref{mapping}. 
In order to see that the second term has the properties in Lemma \ref{negligible}, let us check 
that $\tilde \chi$ also satisfies the properties in \eqref{chi} for $\alpha$. 
Indeed, $\tilde \chi$ vanishes for small $|\theta|$. Moreover, 
suppose that $(x,x',\theta) \in \alpha^{-1}(T^*_0M\times \{|\xi'|\ge \varepsilon\})$. 
Then $\partial_{\theta}\varphi(x,x',\theta) =0$, so that in particular 
$\partial_{\theta''}\varphi(x,x',\theta) =0$, which implies that $\theta'' = s(x,x',\theta')$. 
Hence $\tilde \chi\equiv 1$ in a neighborhood of $(x,x',\theta)$, which is 
conic at infinity due to the $1$-homogeneity of $s$.
This concludes the proof.
\end{proof}

We can summarize the previous discussion in the following proposition. 

\begin{proposition}\label{fio1}
$($Correctness of the definition$)$  The operator family $\Phi_{h,\varepsilon,N}$ with 
the Schwartz kernel \eqref{ker2} is independent of the choice of the representation 
\eqref{ker1} and the function $\chi$  modulo sums of $O(h^{N})$ operator families 
and operators whose Schwartz kernels 
have the properties described in Lemma \ref{negligible}. 
\end{proposition}

\paragraph{\bf Compositions.} 
We next study the behavior of the semiclassical
operators associated with a quantized canonical transformation under compositions. 
Suppose that $\Phi_1$ and $\Phi_2$ are quantized canonical transformations and 
$\Phi=\Phi_1\Phi_2$. 
For a semiclassical symbol $a\in \mathbb A'/\mathbb A'_N$ we study the difference 
\begin{equation*}
 \Phi_{h,\varepsilon,N}Op_h(a)- \Phi_{1,h,\varepsilon,N} \Phi_{2,h,\varepsilon,N}Op_h(a) .
\end{equation*} 
Denote by $C_1$ and $C_2$ the canonical relations associated with $\Phi_1$ and
$\Phi_2$. 
We recall that there exist changes of coordinates  on $M$ 
such that $C_1$ and $C_2$, respectively, 
can be written with phase functions of the form 
\begin{eqnarray*}%\lefteqn{}\\
\varphi_1(x,y,\theta) &=& S(x,\theta) -y\cdot \theta; \\
\varphi_2(y,x',\tau) &=& y\cdot\tau- T(x',\tau),
\end{eqnarray*}
with suitable functions $S$ and $T$, homogeneous of degree 1 in $\theta$ and $\tau$,
respectively (note that this implies that $d=n$), see \cite[Proposition 25.3.3]{HIV}.
We also know from Proposition \ref{fio1} that this transition changes the associated semiclassical operators at most by   $O(h^{N})$-families 
and families with the properties in Lemma \ref{negligible} for the corresponding canonical 
relations. So we consider the representations  
\begin{eqnarray}\label{kPhi1}
K^{ \Phi_1}(x,y) &=& \int e^{i(S(x,\theta) - y\cdot \theta)}  b^1(x,y,\theta) \, d\theta,\\
K^{ \Phi_2}(y,x') &=& \int e^{i(y\cdot \tau - T(x',\tau) )}  b^2(y,x',\tau) \, d\tau\label{kPhi2}.
\end{eqnarray}
The kernel $K^{\Phi}$ of the composition $\Phi= \Phi_1 \Phi_2$ then is given by 
\begin{eqnarray}\nonumber%\lefteqn{}\\
K^{\Phi}(x,x') &=&   \int K^{\Phi_1}(x,y) K^{\Phi_2}(y,x')\, dy \\
&=&\int\!\!\iint e^{i(S(x,\theta)-T(x',\tau) + y\cdot(\tau- \theta))}b^1(x,y,\theta)
 b^2(y,x',\tau) \, d\tau d\theta dy.\label{KPhi}
\end{eqnarray}

\begin{proposition}\label{fio2}
{\rm (Composition formula)} Given quantized canonical transformations $\Phi_1,\Phi_2$ associated with the canonical transformations $C_1$ and $C_2$, and $\Phi=\Phi_1\Phi_2$, we have
\begin{equation*}
 \Phi_{h,\varepsilon,N}Op_h(a)=\Phi_{1,h,\varepsilon,N} \Phi_{2,h,\varepsilon,N}Op_h(a) \mod O(h^{N})\text{-families}
\end{equation*} 
for any semiclassical symbol $a\in\mathbb{A}'/\mathbb{A}'_N$, provided that $\varepsilon>0$ in \eqref{chi} is chosen such that
$a$ vanishes on the subsets 
$\{|\xi|<2\varepsilon\},{C_2}^{-1}\{|\xi|<2\varepsilon\}\subset T^*_0M$.
A corresponding statement  is valid, if we take the product with $Op_h(a)$ on the left and choose 
$\varepsilon$ appropriately. 
\end{proposition}

\begin{proof}
Going over to adjoints and noting that these have the inverse canonical relations, we see that it is sufficient to study the case, where $Op_h(a)$ acts from the right. 

1. Let us first assume that the kernels of $\Phi_{1,h,\varepsilon,N}$ and $\Phi_{2,h,\varepsilon,N}$ are derived from the representations   \eqref{kPhi1} and \eqref{kPhi2},
continuing the discussion above. According to Lemma \ref{mapping},  neglecting terms of order $\le -N$ in the asymptotic expansion of the amplitude results in errors which are $O(h^N)$-families. 
Composition with $Op_h(a)$ then also furnishes $O(h^N)$-families.  

So let us fix excision functions $\chi_1$ and $\chi_2$ satisfying condition \eqref{chi} for 
$C_1$ and $C_2$, respectively. 
As pointed out after \eqref{chi}, we may assume that $\chi_1(x,y,\theta) \equiv 1$ 
and $\chi_2(y,x',\tau)\equiv 1$ for large $|\theta|$ and $|\tau|$, respectively. 
The kernel of $\Phi_{1,h,\varepsilon,N}\Phi_{2,h,\varepsilon,N}$ is 
\begin{align*}
K(x,x') =\ &\ \int\!\!\iint e^{\frac ih (S(x,\theta) -T(x',\tau) + y(\tau-\theta))} 
\sum_{j,k<N}h^{j+k-2n} b^1_j(x,y,\theta)  b^2_k(y,x' ,\tau) \\
&\times \chi_1(x,y,\theta)\chi_2(y,x',\tau) \, dyd\tau d\theta
\end{align*}
with symbols $b^1_j$ and $b^2_k$, positively homogeneous of degrees $-j$ and $-k$,
respectively.
 
As in the proof of Lemma  \ref{mapping}, we choose a function $H=H(\theta, \tau)$,
positively homogeneous of degree zero,  
such that $H(\theta, \tau) = 1$ when $|\tau|/2\le |\theta|\le 2|\tau|$ and $H$ has 
support in $\{ |\tau|/3\le |\theta|\le 3|\tau|\}$,  
We split the amplitude into the two terms 
\begin{align*}%\lefteqn{}\\
c(h;x,x',y,\tau, \theta) =& \sum_{j,k<N}h^{j+k-2n} b^1_j(x,y,\theta)  b^2_k(y,x' ,\tau)  \chi_1(x,y,\theta)\chi_2(y,x',\tau)H(\theta,\tau)\\
r(h;x,x',y,\tau, \theta) =& \sum_{j,k<N} h^{j+k-2n}b^1_j(x,y,\theta)  b^2_k(y,x' ,\tau)  \chi_1(x,y,\theta)\chi_2(y,x',\tau)(1-H(\theta,\tau)).
\end{align*}
Again, the term associated with $r$ induces 
an $O(h^\infty)$-family via an integration by parts with respect to $y$. 
The terms in the amplitude $c$ with $j+k\ge N$ furnish $O(h^N)$-families, so 
they can be ignored in the sequel and we can restrict the summation to $j+k<N$.
In order to bring the integral to a standard form, we moreover make the coordinate transform 
$$w=(w_1,w_2,w_3) = (y|\theta,\tau|, \theta,\tau)\in \mathbb R^{3n}$$ suggested by H\"ormander in the proof of Theorem 25.2.3 in \cite{HIV} and express all items in terms 
of  $w$. This makes the phase $1$-homogeneous in $w$. 
In this context we note that $dy= |w_2,w_3|^{-n}dw_1$.

Next choose an excision function $\chi_3$ satisfying the assumptions \eqref{chi}
for the phase $\varphi$ parametrizing  $C_1\circ C_2$
\begin{eqnarray*}
\varphi(x,x',w) = S(x,w_2) - T(x',w_3) +\frac{w_1}{|w_2,w_3|}(w_3-w_2).
\end{eqnarray*}
Without loss of generality we also assume that $\chi_3(x,x',w)\equiv 1$ for large $|w|$. We then study the composition 
$$((\Phi_1\Phi_2)_{h,\varepsilon,N}-\Phi_{1,h,\varepsilon,N}\Phi_{2,h,\varepsilon,N})Op_h(a).
$$
It has the integral kernel
$$(2\pi)^{-n} \int\!\!\iint e^{\frac ih (\varphi(x,z,w) +(z-x')\cdot\eta)}\tilde d(x,z,w) a(z,\eta) 
\, dwdzd\eta.$$
Here, 
\begin{eqnarray*}\lefteqn{
\tilde d(x,z,w) = \Big(\sum_{j+k<N}h^{j+k-3n}
b^1_j\Big(x,\frac{w_1}{|w_2,w_3|},w_2\Big)  
b^2_k\Big(\frac{w_1}{|w_2,w_3|},z ,w_3\Big)\Big)}\\
&&\times \Big(\chi_3(x,z,w)-\chi_1\Big(x,\frac{w_1}{|w_2,w_3|},w_2\Big)
\chi_2\Big(\frac{w_1}{|w_2,w_3|},z,w_3\Big)\Big)\\
&&\times H(w_2,w_3) |w_2,w_3|^{-n}.
\end{eqnarray*}
Formally, the amplitude does not belong to one of the H\"ormander symbol classes. We can fix this problem  as before by introducing a function $\tilde H= H(w,\eta)$, 
homogeneous of degree zero,  such that $\tilde H(w,\eta) = 1$, when
$\frac{c_1}2 |\eta|\le |w|\le 2c_2|\eta|$
for suitable positive constants $c_1$ and $c_2$,  and $\tilde H$ has support in  
$\{\frac{c_1}3 |\eta|\le |w|\le 3c_2|\eta|\}$. 
Multiplying by $\tilde H$ we obtain an
amplitude in $S^0$, at the expense of 
changing the expression by an $O(h^\infty)$-family, which does not affect the result.

We note that the amplitude $\tilde d\tilde H$  now has compact support that lies 
outside a neighborhood of $w=\eta=0$. 
In fact, 
$\chi_1(x,w_1/|w_2,w_3|,w_2)$, $\chi_2(w_1/|w_2,w_3|,z,w_3)$ and $\chi_3(x,z,w)$ vanish for small $|w|$.
For large values of $|w|$,  
$$H(w_2,w_3)\bigl(\chi_3(x,z,w) -\chi_1(x,w_1/|w_2,w_3|,w_2)\chi_2(w_1/|w_2,w_3|,z,w_3)\bigr)=0$$  
by assumption. 
Moreover, for large values of $|\eta|$ also $|\tau|=|w_3|$ will 
be large, hence the amplitude also vanishes.  

The critical points of the phase function
$$\varphi(x,z,w) + (z-x')\cdot \eta= S(x,w_2) - T(z,w_3) +\frac{w_1}{|w_2,w_3|}\cdot(w_3-w_2)
+ (z-x')\cdot \eta$$ 
are given by 
\begin{eqnarray*}%\lefteqn{}\\
w_2&=&w_3\\
\partial_\theta S(x,w_2)&=& \frac{w_1}{|w_2,w_3|} =\partial_\tau T(z,w_3)\\
\partial_z T(z,w_3) &=&\eta\\
z&=&x'.
\end{eqnarray*}
We claim that the amplitude $\tilde d\tilde H a$ vanishes in a neighborhood of these.
To this end we recall that  $\chi_1(x,y,\theta) = 1$ in a neighborhood of all points $(x,y,\theta)$
with $\partial_\theta S(x,\theta) = y$ and $|\theta|\ge \varepsilon$ 
and that $\chi_2(y,z,\tau)=1$ near all points $(y,z,\tau)$, where $\partial_\tau T(z,\tau)=y$ and 
$|\partial_z T(z,\tau)|\ge \varepsilon$. Finally, $\chi_3(x,z,w) = 1$ near all $(x,z,w)$,
where $\partial_\theta S(x,w_2) =\frac{w_1}{|w_2,w_3|} =   \partial_\tau T(z,w_3) $, 
$w_2=w_3$ and $|\partial_z T(z,w_2) |\ge \varepsilon$.
As a consequence 
$$\chi_3(x,z,w) - \chi_1\Big(x,\frac{w_1}{|w_2,w_3|},w_2\Big)
\chi_2\Big(\frac{w_1}{|w_2,w_3|},z,w_3\Big)$$ 
vanishes near all those
critical points of $\varphi(x,z,w)$, where  $|w_2|\ge\varepsilon$ and 
$|\partial_z T(z,w_3)|\ge \varepsilon$. 
By assumption,  $a(z,\eta)$  vanishes on $\{|\xi|\le 2\varepsilon\}$ and $C_2^{-1}(\{|\xi|<2\varepsilon\})\subset T^*_0M$. 
Now, given $(y,z,\tau) \in \text{Crit}_{\varphi_2}$,  $C_2$ maps $(z,\partial_z T(z,\tau))$ to 
$(y,\tau)$. 
On the critical set of the phase, $(z,\eta)=(z, \partial_z T(z,w_3))= C_2^{-1}(z,w_3)$,
so that $a(z,\eta)$ vanishes whenever $|w_3|<2\varepsilon$ or  $|\partial_z T(z,w_3)|\le
2\varepsilon$. Hence the amplitude indeed vanishes near the critical set, so that the 
integral defines an $O(h^\infty)$- family.  
This shows the proposition for the case of kernels with the representations 
\eqref{kPhi1} and \eqref{kPhi2}.

2. In order to treat the general case, denote by 
$K^{\Phi_1}_{h,\varepsilon,N}$ and $K^{\Phi_2}_{h,\varepsilon,N}$ the Schwartz kernels
derived from the original representations of $K^{\Phi_1}$ and $K^{\Phi_2}$, respectively,
and by $\tilde K^{\Phi_1}_{h,\varepsilon,N}$ and $\tilde K^{\Phi_2}_{h,\varepsilon,N}$ 
those derived from the representations  \eqref{kPhi1} and \eqref{kPhi2}. 
Let $\Phi_{j,h,\varepsilon,N}$ and $\tilde \Phi_{j,h,\varepsilon,N}$, $j=1,2$,  
be the associated operators. 

We already know from Proposition  \ref{fio1} that the difference of the Schwartz kernels
$K^{\Phi_j}_{h,\varepsilon,N}-\tilde K^{\Phi_j}_{h,\varepsilon,N}$,  $j=1,2$, 
has the properties in Lemma \ref{negligible} for the corresponding canonical relations
modulo an $O(h^N)$-family.  
We conclude from Lemma \ref{negligible} and the fact that  $\Phi_{1,h,\varepsilon,N}$ and 
$Op_h(a)$ are  $O(1)$-families by Lemma \ref{mapping} that 
$$\Phi_{1,h,\varepsilon,N}(\Phi_{2,h,\varepsilon,N}-\tilde \Phi_{2,h,\varepsilon,N})Op_h(a),$$
is an $O(h^N)$-family.

It remains to study $(\Phi_{1,h,\varepsilon,N}-\tilde \Phi_{1,h,\varepsilon,N}) \tilde \Phi_{2,h,\varepsilon,N}Op_h(a)$.
The Schwartz kernel of $\tilde \Phi_{2,h,\varepsilon, N} Op_h(a)$  is 
\begin{eqnarray*}\lefteqn{K(y,x') =(2\pi )^{-n} h^{-2n} }\\
&&\times \int\!\!\iint e^{ \frac ih \psi(y,x',z,\eta,\tau)} \ 
\sum_{k<N}h^k b^2_{k}(y,z,\tau)\chi_2(y,z,\tau) 
a(z,\eta) \,
dz d\eta d\tau.
\end{eqnarray*}
Here, the phase is $\psi(y,x',z,\eta,\tau) =  y\cdot \tau - T(z,\tau)+(z-x')\cdot\eta$. It has critical 
points precisely, when $y=\partial_\tau T(x',\tau)$, $z=x'$ and $\eta= \partial_{x'}T(x',\tau)$. 
After multiplication by a cut-off function $H=H(\tau,\eta)$ as above, $|\tau|\sim |\eta|$ on 
${\rm supp}\,H$, and the amplitude belongs to $S^0$, uniformly in $h$. 
We moreover find that 
\begin{itemize}
\item $\chi_2$ vanishes for small $\tau$, hence so does the full amplitude; 
\item  the amplitude vanishes for $|\partial_{x'}T(x',\tau)|<2\varepsilon$, since $\eta=\partial_{x'}T(x',\tau)$.
\end{itemize}

We know that $\Phi_{1,h,\varepsilon,N}-\tilde \Phi_{1,h,\varepsilon,N}$ 
is the sum of an $O(h^{N})$-family and an operator with the properties in 
Lemma \ref{negligible}. When composing it from the right with $\tilde \Phi_{2,h,\varepsilon,N}Op_{h}(a)$, the $O(h^{N})$-family produces an
$O(h^N)$-family. So let us consider the composition with the operator having the properties 
in Lemma \ref{negligible}. We can write it with a kernel of the form 
$$h^{-(d+n)/2}\ \int e^{\frac ih \tilde \varphi(x,y,\theta)} \tilde c(x,y,\theta)\, d\theta,$$
where $\tilde c $ vanishes for small and large $|\theta|$ and furthermore  
vanishes in a neighborhood of all points in ${\rm Crit}_{\tilde \varphi}$, 
for which $|\partial_y\tilde \varphi(x,y,\theta)|\ge \varepsilon$. 
As a consequence, the Schwartz kernel of the composition has the form 
\begin{eqnarray*}
\lefteqn{\tilde K(x,x') = (2\pi)^{-n} h^{-2n-(n+d)/2}}\\
&&\times
\int\!\!\iint\!\!\iint e^{\frac ih \tilde \psi(x,x',y,z,\theta,\tau,\eta)}d(h;x,x',y,z,\theta,\tau,\eta)
\, dyd\theta dz d\tau d\eta,
\end{eqnarray*}
where 
$$\tilde \psi(x,x',y,z,\theta,\tau,\eta) = \tilde \varphi(x,y,\theta) +y\cdot \tau -T(z,\tau)
+(z-x' )\cdot\eta$$
and 
$$d(h;x,x',y,z,\theta,\tau,\eta)) = \tilde c(x,y,\theta) \sum h^kb^2_{k} (y,z,\tau) 
\chi_2(y,z,\tau) a(z,\eta).$$
For the critical points of this phase, we obtain the additional conditions 
$$\tau =- \partial_{y} \tilde \varphi(x,y,\theta), \quad \partial_\theta \tilde \varphi(x,y,\theta)=0.$$
In particular, whenever we are on the critical set of $\tilde \psi$, $(x,y,\theta)$ lies on the 
critical set of $\tilde\varphi$, and so,  $\tilde c (x,y,\theta)=0$, whenever 
$|\partial_y\tilde \varphi(x,y,\theta)|>\varepsilon$. 
On the critical set we moreover have  
$(x',\eta) = (x',\partial_{x'}T(x',\tau)) 
= C_2^{-1}(y, \tau) = C^{-1}_2(y,-\partial_y\tilde \varphi(x,y,\theta))$. 
As  $a(y,\eta) = 0$ on $C_2^{-1}(\{|\xi|\le 2\varepsilon\})$, 
the amplitude vanishes in a neighborhood of all critical points, so that
the composition is an $O(h^\infty)$-family. 
This concludes the proof. 
\end{proof}

\section{Egorov's Theorem for Semiclassical Operators}\label{sec4}

\begin{theorem}\label{Egorov}
{\rm (Semiclassical Egorov theorem)}
Given a semiclassical symbol $a\in\mathbb{A}'/\mathbb{A}'_N$, 
the composition
\begin{equation}\label{op-3}
  \Phi_{h,\varepsilon,N} Op_h(a) \Phi^{-1}_{h,\varepsilon,N},
\end{equation}
where $\Phi^{-1}_{h,\varepsilon,N}$ is a semiclassical quantized canonical transformation associated with $\Phi^{-1},$
is a semiclassical pseudodifferential operator with symbol %equal to
\begin{equation}\label{eq-symb4}
\sigma(\Phi_{h,\varepsilon,N} Op_h(a) \Phi^{-1}_{h,\varepsilon,N})
\equiv \Bigl[1+\sum_{\substack{1\le k< N,\\ 0<|\alpha|+|\beta|\le 2k}}h^k \mu_{k,\alpha,\beta} D^\alpha_x D^\beta_\xi \Bigr](C^{-1})^*a 
\end{equation}
modulo symbols which induce $O(h^N)$-families.
 
Here $\varepsilon$ is chosen such that $a$ vanishes on the sets $\{|\xi|<2\varepsilon\}, C^{-1}\{|\xi|<2\varepsilon\}\subset T^*M$. Moreover, the coefficients $\mu_{k,\alpha,\beta}(x,\xi)$ are homogeneous functions in $\xi$ of  degree $|\beta|-k$,
and are expressed in terms
of the amplitudes and phase functions of $\Phi$ and $\Phi^{-1}$; they do not depend on the 
choice of the cut-off functions $\chi$, $\varepsilon$, etc. 
\end{theorem}

The proof refines (and relies on) the argument given by 
Martinez  \cite[Proof of Proposition 5.5.4]{Mart1} in that we are working with classical symbols and control both the expansion in $h$ and the orders of the symbols.

\paragraph{\bf Preliminaries} For the canonical transformation $C$ consider a point 
$(x_0,\xi_0) = C(y_0,\eta_0)$ in $T^*_0 M $.
Possibly after a change of coordinates, we may assume that 
locally near $(x_0,\xi_0; y_0,\eta_0)$,  the operators $\Phi$ and $\Phi^{-1}$ are given by  
Schwartz kernels $K^\Phi$ and $K^{\Phi^{-1}}$ of the form 
\begin{eqnarray}%\lefteqn{}\\
\label{K1}
K^{\Phi} (x,y) &=& \int e^{i(x\cdot \eta-S(y,\eta))} p(x,y,\eta)\, d\eta, \text{ and }\\
\label{K2}
K^{\Phi^{-1}}(x,y) &=& \int e^{i(S(x,\eta)- y\cdot\eta)} q(x,y,\eta)\, d\eta
\end{eqnarray}
with amplitudes $p\sim\sum p_j$ and $q\sim\sum q_j$ of order zero (since now $d=n$).
Note that in contrast to the notation used before, $S$ here is a generating function 
for $C^{-1}$. We also note that $p(x,y,\eta)$ in \eqref{K1} can be assumed to vanish outside a conical neighborhood 
of $(x_0,y_0,\eta_0)$. Similarly, $q(x,y,\eta)$ in \eqref{K2} vanishes outside a conical neighborhood of $(y_0,x_0,\xi_0)$. 

We write $\tilde \Phi_{h,\varepsilon,N}$ and $\tilde \Phi^{-1}_{h,\varepsilon,N}$
for the operators with the Schwartz kernels 
\begin{eqnarray*}%\lefteqn{}\\
\tilde K_{h,\varepsilon,N}^{\Phi} (x,y) &=&h^{-n} \int e^{\frac ih(x\cdot \eta-S(y,\eta))}\sum_{j<N}h^jp_j(x,y,\eta)\chi_p(x,y,\eta)\, d\eta\\
&=& \int e^{ i(x\cdot \eta-S(y,\eta))} p(h;x,y,\eta)\, d\eta
\text{ and }\\
\tilde K_{h,\varepsilon,N}^{\Phi^{-1}}(x,y) &=&
h^{-n}\int e^{\frac ih(S(x,\eta)-y\cdot \eta)}\sum_{j<N} h^jq_j(x,y,\eta)\chi_q(x,y,\eta)\, d\eta\\
&=& \int e^{i(S(x,\eta)-y\cdot \eta)} q(h;x,y,\eta)\, d\eta,
\end{eqnarray*}
where $\chi_p$ and $\chi_q$ are excision functions for $\Phi$ and $\Phi^{-1}$, 
respectively, with the corresponding properties in \eqref{chi}  and 
\begin{eqnarray*}%\lefteqn{}\\
p(h;x, y,\eta) &=& \sum_{j<N} p_j(x,y,\eta)\chi_p(x,y,h\eta),\\ 
q(h;x, y,\eta) &=& \sum_{j<N} q_j(x,y,\eta)\chi_q(x,y,h\eta).
\end{eqnarray*}
We know by Proposition~\ref{fio1} that  
$\tilde \Phi_{h,\varepsilon,N}-\Phi_{h,\varepsilon,N}$ 
and $\tilde \Phi^{-1}_{h,\varepsilon,N}-\Phi^{-1}_{h,\varepsilon,N}$ are sums of 
$O(h^N)$-families and operators whose Schwartz kernels 
have the properties in Lemma \ref{negligible} for $C$ and $C^{-1}$, respectively. 
Since $a$ is assumed to vanish on the sets $\{|\xi|<2\varepsilon\}$, and $C^{-1}\{|\xi|<2\varepsilon\}\subset T_0^*M$, 
composition with $Op_h(a)$ furnishes  $O(h^N)$-families. 
As the statement only concerns powers of $h$ up to order $N-1$, it will be 
sufficient to work in the sequel with $\tilde \Phi_{h,\varepsilon,N}$
and $\tilde \Phi^{-1}_{h,\varepsilon,N}$.  

\paragraph{\bf The symbol of  $\tilde \Phi_{h,\varepsilon,N} Op_h(a)\tilde \Phi^{-1}_{h,\varepsilon,N}$} 
We suppose that $a(y,\eta)$ is supported in a sufficiently small conic neighborhood of $(y_0,\eta_0)$. 
Then 
\begin{eqnarray*}\lefteqn{Op_h(a) {\tilde \Phi}^{-1}_{h,\varepsilon,N}u(y) }\\
&=& (2\pi)^{-n}  \iint\!\!\iint e^{i(y-y')\cdot \eta}e^{i (S(y',\xi') -x'\cdot \xi')} 
a(y,h\eta) q(h;y',x',\xi') u(x')\, dx'd\xi'dy'd\eta
\end{eqnarray*}
and 
\begin{eqnarray*}%\lefteqn{}\\
\tilde \Phi_{h,\varepsilon,N} Op_h(a) {\tilde\Phi}_{h,\varepsilon,N}^{-1} u(x)
&=&(2\pi)^{-n} \iint e^{i(x-x')\cdot \xi} b(h;x,x',\xi) u(x') \, dx'd\xi,
\end{eqnarray*}
where
\begin{eqnarray*}%\lefteqn{}\\
b(h;x,x',\xi) &=& \iint\!\!\iint e^{i(y-y')\cdot \eta}e^{i(S(y',\xi') -S(y,\xi))} 
e^{ix'\cdot(\xi-\xi')}\\
&&\ \ \times p(h;x,y,\xi) 
a(y,h\eta) q(h;y',x',\xi')  dydy'd\eta d\xi'.
\end{eqnarray*}

Write 
$$S(y,\xi)-S(y',\xi') =M_1(y,y',\xi)(y-y') + M_2(y',\xi,\xi')(\xi-\xi')$$
with 
\begin{eqnarray*}%\lefteqn{}\\
M_1(y,y',\xi) &=& \int_0^1 \partial_yS((1-s)y'+sy,\xi) \, ds;\\
M_2(y',\xi,\xi') &=& \int_0^1 \partial_\xi S (y',(1-s)\xi'+s\xi) \, ds.
\end{eqnarray*}
From the corresponding properties of $S$ we deduce that $M_1$ is one-homogeneous in $\xi$ and $M_2$ is zero-homogeneous 
in $(\xi,\xi')$. 
Since $p$ and $q$ vanish for small $|\xi|$ and $|\xi'|$, respectively, we can assume 
$S$ to be an element of the symbol class 
$S^1$, and $M_1$ and $M_2$ can be assumed to be elements of $S^1$ and $S^0$,
respectively, for $\xi,\xi'$ in small conic neighborhoods of $\xi_0$. 

With the translation $\eta\mapsto \eta+M_1$ we can rewrite $b$ in the form 
\begin{eqnarray*}%\lefteqn{}\\
b(h;x,x',\xi) &=&\iint\!\!\iint e^{i(y-y')\cdot \eta}
e^{i(x'-M_2(y',\xi,\xi'))\cdot(\xi-\xi')}\\
&&\ \ \times p(h;x,y,\xi) 
a(y,h(\eta+M_1(y,y',\xi))) q(h;y',x',\xi')  dydy'd\eta d\xi'.
\end{eqnarray*}
Since $S$ is a generating function for $C^{-1}$, $M_1(y,y',\xi)$ is close to $\eta_0$  
for $y, y'$ close to $y_0$, and $\xi$ close to $\xi_0$.
We conclude that $(y,y',\eta,\xi) \mapsto a(y,\eta+M_1(y,y',\xi))$ then 
is a symbol of order zero.

Next we observe that, for $\xi,\xi'$ in the small conic neighborhood of $\xi_0$ and
$y'$ close to $y_0$, the map 
\begin{eqnarray}\label{M2}%\lefteqn{}\\
y'\mapsto M_2(y',\xi,\xi') 
\end{eqnarray}
is invertible. Indeed, this follows from the fact that $y\mapsto \partial_\xi S(y,\xi_0)$ 
is a diffeomorphism from a neighborhood of $y_0$ to a neighborhood of $x_0$ and
the fact that, due to the $0$-homogeneity of $\partial_\xi S$, 
$$   \partial_\xi S(y,\xi)-  \partial_\xi S (y,\xi_0)=
 \partial_\xi S (y,\xi/|\xi|)- \partial_\xi S (y,\xi_0/|\xi_0|)$$
can be made arbitrarily small by choosing the conic neighborhood of $\xi_0$ small. 
We denote the inverse of \eqref{M2} by $M_{\xi,\xi'}^{-1}$. 
From the above consideration we see that $(z, \xi, \xi')\mapsto M_{\xi,\xi'}^{-1}(z)$
is homogeneous of degree zero in $(\xi,\xi')$; moreover, $(z,\xi,\xi')\mapsto M^{-1}_{\xi,\xi'}(z)$ can be regarded as an element in $S^0$. 
Letting $y'=M_{\xi,\xi'}^{-1}(z')$ and $y=M_{\xi,\xi'}^{-1}(z)$, the expression for $b$ becomes
\begin{eqnarray*}%\lefteqn{}\\
b(h;x,x',\xi) &=& \iint\!\!\iint e^{i(M_{\xi,\xi'}^{-1}(z)-M^{-1}_{\xi,\xi'}(z'))\cdot\eta}
e^{i(x'-z')\cdot(\xi-\xi')}p(h;x,M_{\xi,\xi'}^{-1}(z),\xi)\\
&&\ \times \  
a(M_{\xi,\xi'}^{-1}(z),h(\eta+M_1(M_{\xi,\xi'}^{-1}(z),M_{\xi,\xi'}^{-1}(z'),\xi)))q(h;M_{\xi,\xi'}^{-1}(z'),x',\xi')\\
&&\  \times 
 \ |J_{M^{-1}_{\xi,\xi'}}(z')| \ |J_{M^{-1}_{\xi,\xi'}}(z)|\ dzdz'd\eta d\xi'.
\end{eqnarray*}
where $J_\bullet$ denotes the corresponding Jacobians. For 
$$V(z,z',\xi,\xi') = \int_0^1 \partial_x M^{-1}_{\xi,\xi'}(sz+(1-s)z')\, ds$$
we have 
$$ M^{-1}_{\xi,\xi'}(z) -  M^{-1}_{\xi,\xi'}(z') = 
V(z,z',\xi,\xi')(z-z').$$
Then $V$ is zero-homogeneous in $(\xi,\xi')$ and actually can be considered a symbol in 
$S^0$.  Write 
\begin{eqnarray}\lefteqn{\mbox{\ \ }
b(h;x,x',\xi) = 
\iint\!\!\iint e^{iV(z,z',\xi,\xi')(z-z')\cdot\eta+i(x'-z')\cdot(\xi-\xi')}p(h;x,M_{\xi,\xi'}^{-1}(z),\xi)\label{b}}\\
&&\ \times \  
a(M_{\xi,\xi'}^{-1}(z),h(\eta+M_1(M^{-1}_{\xi,\xi'}(z),M^{-1}_{\xi,\xi'}(z'),\xi)))\ q(h;M_{\xi,\xi'}^{-1}(z'),x',\xi')
\nonumber\\
&&\  \times 
 \ |J_{M^{-1}_{\xi,\xi'}}(z')| \ |J_{M^{-1}_{\xi,\xi'}}(z)|\ %|\det V(z,z',\xi,\xi')|\ 
 dzdz'd\eta d\xi'
 \nonumber\\
 &=& \iint\!\!\iint e^{i(z-z')\cdot\sigma+i(x'-z')\cdot(\xi-\xi')}\tilde p(h;x,z,\xi,\xi')
 \nonumber\\
&&\ \times \  
\tilde a(h;z,z',\sigma,\xi,\xi')\tilde q(h;z',x',\xi,\xi')
\ \tilde J(z,z',\xi,\xi')\ 
dzdz'd\sigma d\xi',
\nonumber 
\end{eqnarray}
where 
\begin{eqnarray*}%\lefteqn{}\\
\tilde p(h;x,z,\xi,\xi')&=& p(h;x,M_{\xi,\xi'}^{-1}(z),\xi), \\
\tilde a(h;z,z',\sigma,\xi,\xi')&=&a(M_{\xi,\xi'}^{-1}(z),
 h(V^t(z,z',\xi,\xi')^{-1}\sigma+M_1(M_{\xi,\xi'}^{-1}(z),M_{\xi,\xi'}^{-1}(z'),\xi))),\\
\tilde q(h;z',x',\xi,\xi')&=&q(h;M_{\xi,\xi'}^{-1}(z'),x',\xi'),\\
\tilde J(z,z',\xi,\xi')&=&|J_{M^{-1}_{\xi,\xi'}}(z')| \ |J_{M^{-1}_{\xi,\xi'}}(z)|\ 
|\det V(z,z',\xi,\xi')|^{-1}.
\end{eqnarray*}

\paragraph{\bf The asymptotic expansion of the symbol}  
We shall see that this indeed furnishes the desired expansion \eqref{eq-symb4}. 
In order to sketch the idea let 
$$c(h;z,z',x,x',\sigma, \xi,\xi')= \tilde p(h;x,z,\xi,\xi') 
\tilde a(h;z,z',\sigma,\xi,\xi')\tilde q(h;z',x',\xi,\xi')
\ \tilde J(z,z',\xi,\xi').$$ 
In a first step, we apply a Taylor expansion in $z$ at $z=z'$:
\begin{eqnarray*}%\lefteqn{
c(h;z,z',x,x',\sigma, \xi,\xi') 
&=& \sum_{|\alpha|<N}
\frac{1}{\alpha!} \partial_z^\alpha c(h;z,z',x,x',\sigma, \xi,\xi')_{|z=z'}(z-z')^\alpha, \\
&&+ r_N(h;z,z',x,x',\sigma, \xi,\xi')
\end{eqnarray*}
with 
\begin{eqnarray}\lefteqn{\mbox{\ \ \ \ }
\label{rN}
r_N(h;z,z',x,x',\sigma, \xi,\xi')}\\
&=& N\ \sum_{|\gamma|=N}
\frac{(z-z')^\gamma}{\gamma!} 
\int_0^1 (1-s)^{N-1}\partial^\gamma_zc(h;z'+s(z-z'),z',x,x',\sigma,\xi,\xi')\, ds.
\nonumber\end{eqnarray}
We can then decompose $c=c_1+c_2$, where $c_1$ contains the terms 
from the expansion and $c_2$ those from the remainder. 

Let us first consider $c_1$. 
Integration by  parts together with the evaluation of the oscillatory integral over $z$ and 
$\sigma$ shows that 
\begin{eqnarray}
\lefteqn{\nonumber
 \iint\!\!\iint e^{i(z-z')\cdot\sigma+i(z'-x')\cdot(\xi'-\xi)}
\partial^\alpha_z c(h;z,z',x,x',\sigma, \xi,\xi')_{|z=z'}(z-z')^\alpha\ 
 dzdz'd\sigma d\xi'}
\nonumber\\
 &=& \!\!\iint e^{i(z'-x')\cdot(\xi'-\xi)}
\partial^\alpha_z D^\alpha_\sigma c(h;z,z',x,x',\sigma, \xi,\xi')_{|z=z',\sigma=0}
 \,dz'd\xi'.
 \nonumber %\label{OscInt1}
\end{eqnarray}
Assuming for the moment that the expansion makes sense, we can iterate the procedure 
by applying a Taylor expansion up to order $N-1$ in $z'$ at $z'=x'$. 
This yields an expansion for $b$ of the form 
\begin{eqnarray*}%\lefteqn{}\\
b(h;x,x',\xi)\sim\sum_{|\beta|<N} \frac1{\beta!} \partial_{z'}^\beta D^\beta_{\xi'}
\Big(\sum_{|\alpha|<N}\frac1{\alpha!}    
\partial^\alpha_z D^\alpha_\sigma c(h;z,z',x,x',\sigma, \xi,\xi')_{|z=z',\sigma=0}
\Big)_{|z'=x', \xi'=\xi} 
\end{eqnarray*}  
up to the corresponding remainder terms.
Assuming also that $b$ is an amplitude in $S^0$, we find a corresponding symbol 
$\tilde b =\tilde b(h;x,\xi)$; it has the expansion 
$$\tilde b(h;x,\xi) \sim 
\sum_{\gamma }
\frac1{\gamma!}\partial^\gamma_{x'} D^\gamma _\xi b(h;x,x',\xi)_{|x=x'}.$$
Summing up we expect that  
\begin{eqnarray}\lefteqn{\label{expansion}\mbox{\ \ \ }
\sigma(\Phi_{h,\varepsilon,N} Op_h(a) \Phi^{-1}_{h,\varepsilon,N})\sim\sum_{|\alpha|,|\beta|,|\gamma|<N}\frac 1{\alpha!\beta!\gamma!} }\\
&&
\times D^\gamma_{\xi}\partial^\gamma_{x'} 
\Big( D^\beta_{\xi'}\partial^\beta_{z'}
\Big( D^\alpha_{\sigma}\partial^\alpha_{z} 
\Big(c(h; z,z',x,x',\sigma, \xi,\xi')\Big)_{|z=z',\sigma=0}
\Big)_{|z'=x', \xi'=\xi}\Big)_{|x'=x}
\nonumber
\end{eqnarray}
modulo symbols which induce $O(h^N)$-families.

It now remains to check two facts: 
\begin{itemize}
\item The terms in this expansion are actually of the form in \eqref{eq-symb4}
\item  The remainder terms give $O(h^N)$-families. 
\end{itemize} 

\paragraph{\bf The terms  in the expansion.}
We first note that 
\begin{eqnarray*}\lefteqn{
D_{\sigma} \tilde a (h;z,z',\sigma, \xi,\xi')_{|\sigma=0}}\\
&=&h\  D_{\eta} a\bigl(M_{\xi,\xi'}^{-1}(z), 
h M_1(M_{\xi,\xi'}^{-1}(z),M_{\xi,\xi'}^{-1}(z'),\xi)\bigr) (V^t(z,z',\xi,\xi'))^{-1}.
\end{eqnarray*}
Corresponding formulae hold for other derivatives with respect to the various variables.
Recalling that $(y,y',\xi)\mapsto M_1(y,y',\xi)$ and $(z,\xi,\xi') \mapsto M^{-1}_{\xi ,\xi'}(z) $
are actually symbols in $S^1$ and $S^0$, respectively, and that we can write, for example, 
$$\partial_\xi (M_{\xi,\xi'}^{-1} (z))= 
h  \partial_{\tilde \xi} (M_{\tilde \xi,\tilde \xi'}^{-1}(z))_{|\tilde \xi=h\xi, 
\tilde\xi' = h\xi'}$$
we find by induction that 
\begin{eqnarray*}%\lefteqn{}\\
D^{\alpha}_{\xi}D^{\alpha'}_{\xi'}D^{\beta}_{z}D^{\beta'}_{z'}D^{\gamma}_{\sigma}
\tilde a (h;z,z',\sigma, \xi,\xi')_{|\sigma=0}=  
\breve a(z,z', h\xi, h\xi') h^{|\alpha|+|\alpha'| + |\gamma|},
\end{eqnarray*}
where  $\breve a$ is a symbol of order $-|\alpha|-|\alpha'|-|\gamma|$. 

Next consider derivatives of  
\begin{eqnarray}\label{tilde_p}%\lefteqn{}\\
\tilde p(h; x,z',\xi,\xi') = \sum_{j<N}p_j(x,M^{-1}_{\xi,\xi'}(z'),\xi)\chi_p(x,M^{-1}_{\xi,\xi'}(z'),h\xi)
\end{eqnarray}
In view of the above observations, 
$$(x,z',\xi,\xi')\mapsto \sum_{j<N}p_j(x,M^{-1}_{\xi,\xi'}(z'), \xi)$$
defines a symbol of order zero off the zero section.

A corresponding argument applies to   $\tilde q$, and also 
$(z,z',\xi,\xi') \mapsto \tilde J(z,z',\xi,\xi')$ is seen to be a symbol of order zero smooth off the zero section. 

This shows that the expansion \eqref{expansion} contains the symbols of decaying order and increasing powers 
of $h$. 
We notice additionally that derivatives with respect to 
$\sigma$ always produce powers of $h$ together with decay; 
derivatives with respect to $\xi$ and $\xi'$ lower the order, but produce powers of $h$ 
only if they fall on $\tilde a$.  

As a consequence, we also obtain a corresponding expansion for the terms on the 
right hand side of \eqref{expansion}.
More is true:  For $\xi=\xi'$, we have $M_2(y',\xi,\xi) =  \partial_\xi S(y',\xi)$. 
Since $S$ is by construction a generating function for $C^{-1}$,  $M^{-1}_{\xi,\xi}(x)$ 
is the base point component of $C^{-1}(x,\xi)$ and thus 
$M_1(M^{-1}_{\xi,\xi}(x), M^{-1}_{\xi,\xi}(x ), \xi)$ 
is the component in the fiber, so that $a$ and its  derivatives in \eqref{expansion} are evaluated at 
$C^{-1}(x,h\xi)$.

We claim that in \eqref{expansion} after the substitutions $z=z'=x=x', \xi=\xi',\sigma=0$
we have $\chi_p= \chi_q\equiv 1$ on the support of $a$ so that all terms with derivatives of 
$\chi_p$ and $\chi_q$ are equal to zero
and in the remaining terms these functions can be replaced by $1$. 
Indeed, $a$ and its derivatives  in \eqref{expansion} 
are evaluated at the point $C^{-1}(x,h\xi) = (y,\partial_y S(y,h\xi))$, where $y$ is the solution of $\partial_\xi S(y,\xi)=x$.
Thus, $a$ vanishes, whenever we have 
\begin{eqnarray}\label{eq-f1}%\lefteqn{
 |\partial_y S(y,h\xi)|\le 2\varepsilon %}
\end{eqnarray}
by the assumption in our theorem. 
Moreover, $\chi_p$ and its derivatives in \eqref{expansion} are evaluated 
at $(x,y,h\xi)$ and it is identically equal to $1$ in a neighborhood of the set 
$$
\{(x,y,\xi)\;|\; \alpha(x,y,h\xi) \text{ satisfies }|\xi'|\ge \varepsilon\},
$$
see \eqref{chi}. 
But in this case we have the phase function $\varphi(x,y,\xi)=x\cdot \xi-S(y,\xi)$.
Hence 
$$
\alpha(x,y,h\xi)=(x,\partial_x\varphi, y,-\partial_y\varphi) = (x,\xi,y,\partial_yS(y,\xi)).
$$
Thus, $\chi_p$ is identically equal to $1$, whenever we have 
\begin{eqnarray}\label{eq-f2}%\lefteqn{
 |\partial_y S(y,h\xi)|\ge \varepsilon. %}
\end{eqnarray}
Thus, from \eqref{eq-f1} and \eqref{eq-f2} we obtain the desired statement that $\chi_pa=a$ in \eqref{expansion}. 
The proof of the identity $a\chi_q=a$ is similar. 

Recall now that, in order to obtain the symbol of the composition as a 
{\em semiclassical} symbol,
we have to undo the scaling of the covariable by $h$ and replace $h\xi$ by $\xi$. 
Hence the derivatives of $a$ are evaluated at $C^{-1}(x,\xi)$, those of $p$ and $q$ 
at $C^{-1}(x,\xi/h)$.  
As $p_j(x,y,\xi/h) = h^{j}p_j(x,y,\xi)$ with corresponding relations for the derivatives, 
each derivative of $p$ with respect to $\xi$ or $\xi'$ will contribute a factor $h$
after rescaling. 

It is well-known that the leading term in the expansion is $a(C^{-1}(x,\xi))$.
It  can be determined by considering  
the contribution for $|\alpha|=|\beta|=|\gamma|=0$. 
In view of the fact that $p$ and $q$ are the amplitudes in the Schwartz kernels 
\eqref{K1} and \eqref{K2} of $\Phi$ and $\Phi^{-1}$, we have 
$$p(x,M_{\xi,\xi}^{-1}(z),\theta)q(M_{\xi,\xi}^{-1}(z),x,\theta)
|\det \partial_z M_{\xi,\xi}^{-1}(z)|=1. 
$$
This shows the assertion.

\paragraph{\bf The remainder terms.}
In order to understand the contribution of the remainder term, consider first one of the terms
in the summation on the right hand side of \eqref{rN}. After an integration by parts, 
its contribution to the amplitude $b$ is given by 
\begin{eqnarray*}\lefteqn{
\frac N{\gamma!} 
\iint \iint\int_0^1 e^{i(z-z')\sigma + i(z'-x')(\xi-\xi')} 
(1-s)^{N-1}}\\
&&D^\gamma_\sigma\partial^\gamma_zc(h;z'+s(z-z'),z',x,x',\sigma,\xi,\xi')\, ds
dz d\sigma dz'd\xi'.
\end{eqnarray*}
Now we see from (an analog of)  \cite[Lemma II.2.4]{K} that the inner three integrals 
furnish an amplitude whose symbol seminorms can be estimated as before by those for 
$\tilde p$, $\tilde a$, and $\tilde q$. In particular, since we take $N$ derivatives with respect 
to $\sigma$, we obtain $N$ powers of $h$ and symbol order $-N$. 
The corresponding consideration holds for the Taylor expansion in the next step. 
Hence the remainder terms preserve the asymptotic expansion found above. 
%\end{proof}

The proof of Egorov's theorem is now complete.

\section{Trace Asymptotics for Fourier Integral Operators}\label{sec5}

Consider a semiclassical Fourier integral operator $\Phi_h$ associated with a homogeneous Lagrangian 
manifold $L\subset T^*_0(M\times M)$ with a Schwartz kernel equal to 
\begin{equation}\label{eq-fio1}
\Phi_h(x,x')=h^{-d/2-n/2}\int e^{\frac i h\varphi(x,x',\theta)}a(x,x',\theta)d\theta, \quad \theta\in \mathbb{R}^d, 
\end{equation} 
where we suppose that the support of the amplitude $a$ is sufficiently small such that in a neighborhood of this
support $\varphi$ is a nondegenerate phase function, which parametrizes $L$. In particular,  the critical set
\eqref{eq-crit1}
is a smooth submanifold and we have  a local diffeomorphism $\alpha$ as in \eqref{eq-a1},
defined in a small conical neighborhood of a point in $L$.

\begin{theorem}\label{lem3}
Suppose  that $L$ and the diagonal $\Delta=\{(x,p;x,p)\}\subset T^*_0(M\times M)$ intersect cleanly
and $a\in S^{m+(n-d)/2}$, where $m+n/2+d/2<0$. Then the trace $\tr \Phi_h$ exists and
admits an asymptotic expansion as $h\to 0$
\begin{equation}\label{eq-5}
 \tr(\Phi_{h})\sim h^{-\dim (L\cap\Delta)/2}\sum_{j\ge 0} \alpha_jh^j, \quad \text{ where } \alpha_j=\int_{L\cap\Delta}m_j .
\end{equation}
Here the   smooth densities $m_j$ are identically equal to zero in a small neighborhood of the
zero section and the integral converges absolutely at infinity. Moreover,  locally $m_j$ can be expressed in terms of a finite number of derivatives of the phase and the amplitude.  
\end{theorem}

\begin{remark}
This result is  close to Theorem 2 in~\cite{StSh78}, where a formula for the  leading term of the asymptotic is stated,
and a sketch of the proof is given. For the algebraic indices, we need the existence of a full asymptotic expansion 
and not just the leading term. Moreover, it turns out that special   estimates at infinity in
$L\cap \Delta$ are necessary, since the critical set and the support of the amplitude 
function are noncompact so that  we can not apply  the stationary phase method directly.
\end{remark}
\begin{proof}
  
1.  
The integral in \eqref{eq-fio1} is absolutely convergent for any given $h$
by our assumption on the order of the amplitude (since $m+(n-d)/2<-d$), and is an oscillatory integral as $h\to 0$. 
Since the integral in \eqref{eq-fio1} is absolutely convergent and depends continuously on $x,x',$ it follows that
$\Phi_h$ is of trace class and its trace is equal to
\begin{equation}\label{eq-tr3}
 \tr\Phi_{h}=h^{-d/2-n/2}\iint e^{\frac i h \varphi(x,x,\theta)}a(x,x,\theta) dxd\theta. 
\end{equation}
Unfortunately, we can not compute this integral directly by the
stationary phase method, since  the support of the amplitude is noncompact and so is 
the set of stationary points.

2.  Let us compute the stationary point set in \eqref{eq-tr3}. Denote the phase  by 
$\psi(x,\theta)=\varphi(x,x,\theta)$.  We identify the
set of stationary points of $\psi$ and the set
\begin{equation}\label{eq-cc1}
 \mathcal{C}=\{(x,x,\theta)\;|\; \partial_\theta \psi(x,\theta) =0\}\subset \crit_\varphi
\end{equation}
using the mapping  $(x,\theta)\mapsto (x,x,\theta)$.
Since locally we have equality of the sets $\alpha(\mathcal{C})=L\cap\Delta$ 
(this follows from \eqref{eq-a1} and \eqref{eq-cc1}) and the intersection of $L$ and $\Delta$ is clean, we obtain that $\mathcal{C}$ is a submanifold.

3. Let us now show that the phase in \eqref{eq-tr3} is nondegenerate in all directions transverse to $\mathcal{C}$. This follows from the following lemma.
\begin{lemma}
Given a point $(x_0,\theta_0)\in\mathcal{C}$, we have 
$$
T_{x_0,\theta_0}\mathcal{C}=\ker \Hess_{x_0,\theta_0}\psi,
$$
where $\Hess_{x_0,\theta_0} \psi$ is the Hessian of the function $\psi$ evaluated at the  point $(x_0,\theta_0).$
\end{lemma}
\begin{proof} 
In this proof, we deal with tangent spaces at the points $(x_0,x_0,\theta)\in\mathcal{C}$ and 
$\alpha(x_0,x_0,\theta)\in L$. For brevity, we omit these points in the notation.

Note that $TC_\varphi=\ker d(\partial_\theta\varphi)$, since $\varphi$ is a nondegenerate phase function.  Hence, 
$$  
TL=\alpha_*(TC_\varphi) 
=\Bigl\{(X,d(\partial_{x}\varphi)(X,X',\Theta), X',-d(\partial_{x'}\varphi)(X,X',\Theta))\;|\; (X,X',\Theta)\in TC_\varphi\Bigr\}.
$$
Since $L$ and $\Delta$ intersect cleanly,  we have $T(L\cap \Delta)=TL\cap T\Delta$, and we get
$$
T(L\cap \Delta)=
\left\{\bigl(X,d(\partial_{x}\varphi)(X,X,\Theta), X,d(\partial_{x}\varphi)(X,X,\Theta)\bigr)\;
\left|\; 
\begin{array}{c}d(\partial_\theta\varphi)(X,X,\Theta)=0\text{ and}\\
d(\partial_{x}\varphi+\partial_{x'}\varphi)(X,X,\Theta)=0
\end{array}
 \right.
 \right\}.
$$
Hence, $T\mathcal{C}=(\alpha_*)^{-1}T(L\cap \Delta)$ is equal to 
$$
 T\mathcal{C}=\{(X,\Theta)\;|\; (X,\Theta)\in \ker d (\partial_\theta\psi); (X,\Theta)\in \ker d(\partial_x\psi )\}.
$$
This equality proves the lemma, since its right hand side is equal to the kernel of the Hessian of $\psi$.
\end{proof}

3. To apply the stationary phase method, write the integral \eqref{eq-tr3} in spherical 
coordinates  
$$
 \theta=r\omega, \qquad r>0, |\omega|=1,\qquad  d\theta=r^{d-1}drd\omega,
$$
where $d\omega$ is the volume form on $\mathbb{S}^{d-1}$.
Due to the homogeneity of the phase function 
\begin{eqnarray}\label{eq-tr4}\tr\Phi_{h}&=&h^{-d/2-n/2}\iiint e^{\frac {ir} h \varphi(x,x,\omega)}a (x,x,r\omega) r^{d-1} dx drd\omega\\
& \equiv& h^{-d/2-n/2} \int_0^\infty  r^{d-1} I(r,h')dr, \nonumber
\end{eqnarray}
where $h'=h/r$ and
\begin{equation}\label{eq-tr5}
 I(r,h')= \iint e^{\frac {i} {h'} \varphi(x,x,\omega)}a (x,x,r\omega) dxd\omega.
\end{equation}
This is an oscillatory integral with  parameter $h'\to 0$. Note also that its amplitude depends
on an additional positive parameter $r>\varepsilon$, where $\varepsilon$ is chosen such that $a(x,x,r\omega)\equiv 0$, whenever $r<\varepsilon$. Since the phase function in the latter integral
is independent of $r$, the stationary set of the integral \eqref{eq-tr5} also does not depend on $r$
and we can obtain uniform estimates in $r$ of this integral using the stationary phase method, see e.g.~\cite{Zwo1}.
Note that we can apply the stationary phase method in this situation, when we have a smooth compact  manifold of stationary points of the phase and the Hessian of the phase
is nondegenerate in the normal directions to  this submanifold. 
Indeed, since the phase function in \eqref{eq-tr5} is just the restriction of the phase function in \eqref{eq-tr3}
to the sphere $|\theta|=1$, it follows that this phase function is nondegenerate on the critical set 
$\mathcal{C}_0=\mathcal{C}\cap \{|\theta|=1\}$, where $\mathcal{C}$ was defined in  \eqref{eq-cc1}. 

Thus, we obtain the asymptotic expansion as $h'\to 0$
\begin{equation}\label{eq-exp4}
I(r,h')\sim \sum_{j\ge 0} h'^{j+\codim \mathcal{C}_0 /2}\int_{\mathcal{C}_0} \bigl[A_{2j} a(x,x,r\omega)\bigr]\vol_{\mathcal{C}_0}  , 
\end{equation}
from the stationary phase formula, where $A_{2j}=A_{2j}(x,\omega,D_x, D_ \omega)$ are linear differential operators
with smooth coefficients of order $\le 2j$ and $\vol_{\mathcal{C}_0}$ denotes a volume form on $\mathcal{C}_0$.
This formula has no oscillatory exponential factors, since the phase is equal to zero on $\mathcal{C}_0$:
indeed,  $\varphi(x,x,\theta)$ is expressed in terms of the differential at this point by Euler's formula, 
but the differential is zero, hence, so is the phase.
Substituting \eqref{eq-exp4} in \eqref{eq-tr4}, we formally obtain the desired asymptotic expansion \eqref{eq-5}
\begin{multline}\label{eq-m3}
\tr(\Phi_{h})=  h^{-d/2-n/2} \int_0^\infty  r^{d-1} I(r,h')dr\sim\\
\sim \int_0^\infty   \sum_{j\ge 0}r^{d-1-j-\codim \mathcal{C}_0/2} h^{-d/2-n/2+j+\codim \mathcal{C}_0/2}
 \left(
   \int_{\mathcal{C}_0} A_{2j} a(x,x,r\omega) \vol_{\mathcal{C}_0} 
 \right)dr=\\
 =h^{-(\dim L\cap\Delta )/2}\sum_{j\ge 0} h^{j}\int_{\mathcal{C}} r^{d-1-j-\codim \mathcal{C}_0/2} A_{2j}a(x,x,r\omega) \vol_{\mathcal{C}_0}dr .
\end{multline}
To complete the proof of Theorem~\ref{lem3},  it suffices  to show that all integrals in \eqref{eq-exp4}  converge absolutely
and to estimate the error terms in the asymptotic expansions. First,  the convergence    at $r=0$  is trivial, since the amplitude is identically zero in a neighborhood of the zero section $\{\theta=0\}$. Second, one checks by 
explicit differentiation  that 
$$
 |A_{2j} a(x,r\omega)| \le C r^{m+(n-d)/2} \quad \text{uniformly in }r\ge\varepsilon
$$
Hence, the integrand in the $j$-th term in  \eqref{eq-m3} is of the order 
$$
 O(r^{d-1+m +(n-d)/2-j-\codim \mathcal{C}_0/2})\le C r^{d-1+m+(n-d)/2}
$$ 
and its integral  with respect to $r$ absolutely
converges, since $m+n/2+d/2<0$ by  assumption. Finally, we estimate the remainders in the asymptotic expansions. 
We use the estimate of the difference between $I(r,h')$ and the first $N$ terms in the sum \eqref{eq-exp4}, see~\cite[Theorem 3.16]{Zwo1}, and obtain that this difference  is bounded by an expression of the form
$$  
C_N h'^{N+\codim \mathcal{C}_0/2}\iint  \max_{(x,x,r\omega)\in \supp a, |\alpha|+|\beta|\le 2N+\codim\mathcal{C}_0+1} \left|\frac{\partial^{\alpha+\beta}}{\partial x^\alpha
 \partial\omega^\beta} a(x,x,r\omega)\right|dxd\omega
$$ 
This expression is of the order ${h'}^{N+\codim\mathcal{C}_0/2}r^{m+(n-d)/2}$ for $r$ large. Integration of  this estimate with respect to $r$ shows that \eqref{eq-m3} is indeed an asymptotic expansion as $h\to 0$.

The proof of Theorem~\ref{lem3} is now complete. 
\end{proof}

\paragraph{\bf Application to quantized canonical transformations.} 

Let $\Phi$ be a quantized canonical transformation associated with the homogeneous canonical transformation
$$
 C:T^*_0M\to T^*_0M.
$$
And suppose that $C$ is of finite order ($C^k=Id$ for some $k\ge 1$). Then $C$ 
is   {\em nondegenerate} in the sense that its fixed point set (denoted by $T^*_0M^C$) is a smooth submanifold
and  at each point  in $T^*_0M^C$ we have (cf. \cite{AtBo1})
\begin{equation}\label{def11}
 \ker(1-dC)= T(T^*_0M^C),
\end{equation}
where $dC$ is the differential of $C$, i.e., the eigenspace with  eigenvalue equal to $1$ coincides with the tangent space to the fixed point set.
  
A direct computation shows that \eqref{def11} is equivalent  to the condition that the intersection 
$$
 \graph C\cap \Delta\subset T^*(M\times M)
$$ 
of the graph of $C$ and the diagonal in the product is  clean.

Thus, we can apply Theorem~\ref{lem3} to operators of the form 
$$
 Op_h(a )\Phi_{h,\varepsilon,N},
$$
where $a\in\mathbb{A}'$ is  a semiclassical symbol of sufficiently negative order (it suffices to take $\ord a<-2\dim M$), and we obtain the asymptotic expansion
\begin{equation}\label{eq-5rr}
 \tr(Op_h(a)\Phi_{h,\varepsilon,N})\sim h^{-\dim  T^*M^C/2}\sum_{j\ge 0} \alpha_jh^j, \qquad \text{ where } \alpha_j=\int_{T^*M^C}m_j
\end{equation}
in integer powers of $h$ (indeed, by \cite{Fds15}, the fixed point sets $T^*M^C$ are   even-dimensional).  
The coefficients $\alpha_j$ in \eqref{eq-5rr} do not depend on the choice of $\Phi_{h,\varepsilon,N}$ up to $j=N-1$. This is proved using the following lemma. 

\begin{lemma}\label{lemma-trace} 
 \begin{enumerate}
  \item If $a\in S^m, m\le 0$ and $j\ge 0$, then  $h^jOp_h(a)$ is an $O(h^N)$-family, where $N=\min(j,-m)$;
  \item If $A_h$ is an $O(h^N)$-family and $N>\dim M$, then the operator $A_h:L^2(M)\longrightarrow L^2(M)$  is of trace class and we have
\begin{equation}\label{eq-tra1}
 \tr A_h=O(h^{N-\dim M}).
\end{equation}
 \end{enumerate}
\end{lemma}
\begin{proof}
 
1. We set $N=\min(j,-m)$ and obtain
$$
\|Op_h(a)\|_{\mathcal{B}(H^{s}_h(M), H^{s+N}_h(M))}\le \|Op_h(a)\|_{\mathcal{B}(H^{s}_h(M), H^{s-m}_h(M))}\cdot \|Id\|_{\mathcal{B}(H^{s-m}_h(M), H^{s+N}_h(M))}
$$
and the two factors in this formula are obviously uniformly bounded as $h\to 0$. This implies the desired statement.

2.  Let us write $A_h$ as the composition
\begin{equation}\label{comp65}
 A_h=(A_h \Lambda^N_h)\Lambda^{-N}_h, \text{ where }\Lambda^k_h=(h^2\Delta+1)^{k/2}:H^s_h(M)\longrightarrow H^{s-k}_h(M).
\end{equation}
Then the composition $A_h \Lambda^N_h$ is uniformly bounded in $L^2$, while $\Lambda^{-N}_h$ is of order $-N<-\dim M$. Hence, it is of trace class and
therefore the composition $(A_h \Lambda^N_h)\Lambda^{-N}_h$ is of trace class.  Let us now estimate the trace. We have 
\begin{equation}
 |\tr A_h|\le \|A_h\|_1\le  \|A_h \Lambda^N_h\| \|\Lambda^{-N}_h\|_1\le C \tr \Lambda^{-N}_h,
\end{equation}
where $\|B\|$ denotes the norm of operator $B$ and $\|B\|_1$ denotes its trace norm (recall that $\|B\|_1=\tr |B|$) and we use standard properties
of the trace norm (see, e.g. \cite{Shu1}, Proposition D3.7). Thus, it remains to estimate the trace of operator $\Lambda^{-N}_h$.
The principal symbol of this operator is $(h^2\xi^2+1)^{-N/2}$. Therefore, the trace of this operator is estimated by an expression of the form
$$
C\iint \frac{dxd\xi}{(h^2\xi^2+1)^{N/2}}=Ch^{-\dim M}\iint \frac{dxd\xi}{(\xi^2+1)^{N/2}}=O(h^{-\dim M}).
$$
In the last equality we used the fact that $N>\dim M$ by the assumption in our lemma. 

This completes the proof of the lemma.
\end{proof}

%%%%%%%%%%%%%%%%%%%%%%%%

\section{Application to Elliptic $G$-operators}\label{sec6}

\paragraph{\bf Elliptic $G$-operators.}

Let $G$ be a finitely generated discrete group, represented on 
$L^2(M)$ by quantized canonical transformations, i.e., there is a map $g\mapsto \Phi_g$,
$g\in G$, which associates to a group element $g$ a quantized canonical transformation 
$\Phi_g$, such that $\Phi_e=I$ and $\Phi_g\Phi_h = \Phi_{gh}$.
In \cite{SSS2} we considered $G$-operators, i.e. bounded operators  
of the form 
\begin{eqnarray}\label{gop1}%\lefteqn{}\\
D=\sum D_g \Phi_g:L^2(M)\longrightarrow L^2(M),
\end{eqnarray}
where the $D_g$ are pseudodifferential operators of order zero and only finitely many of the 
$D_g$ in the sum are different from zero. Below by a $G$-operator we mean an operator of the form \eqref{gop1} and a choice of coefficients $D_g$.

Egorov's theorem \cite{Ego1} states that for a pseudodifferential operator $A$ with 
principal symbol $\sigma_{pr}(A)$, the operator $\Phi_gA\Phi_g^{-1}$ is again a
pseudodifferential operator with principal symbol $\sigma_{pr}(A)\circ C_g^{-1}$, where 
$C_g$ is the canonical transformation associated with $\Phi_g$.    
It is a consequence of this theorem that the operators of the form $D+K$, where
$D$ is as in \eqref{gop1} and $K\in \mathcal K(L^2(M))$ is compact, form an algebra. 
 
To an operator $D$ as in \eqref{gop1}, more precisely, to this particular representation 
of $D$, we associate a principal symbol, namely the tuple $\{\sigma_{pr}(D_g)\}_{g\in G}$ 
of principal symbols of $D_g$, which can be seen as an element in the maximal $C^*$-crossed product $C(S^*M)\rtimes G$
of the algebra of continuous functions on the cosphere bundle
$S^*M$ of $M$. It turns out that $D$ is a Fredholm operator, if its symbol is invertible in $C(S^*M)\rtimes G$, see \cite[Theorem 1]{SSS2}. 
In general,  the inverse symbol has infinitely many nonzero components and, therefore, it is difficult to write explicity a $G$-operator with this symbol.
In this paper, we shall work in the situation, when the inverse symbol has finitely many components. The general case will be considered elsewhere.
Thus, we introduce the following definition. 

\begin{definition}
A $G$-operator is {\em elliptic}, if its principal symbol
$$
 \sigma_{pr}(D)\in C^\infty(S^*M)\rtimes G 
$$
is invertible in the algebraic crossed product.
\end{definition}
One easily proves that ellipticity implies the Fredholm property in Sobolev  spaces. More precisely,  the following 
 lemma holds. To formulate it, we introduce the algebraic crossed product $\Psi(M)\rtimes G$ of the algebra of classical pseudodifferential operators and $G$
 acting on $\Psi(M)$ by conjugation: $A\in\Psi(M),g\in G\mapsto \Phi_gA\Phi_g^{-1}$. Let $\Psi^m(M)\subset \Psi(M) $ be the space of operators of order $\le m$.
\begin{lemma}
Let $D$ be an elliptic $G$-operator. Then for each $N\ge 1$ there exists a $G$-operator $R \in\Psi^0(M)\rtimes G$ such that  
\begin{equation}\label{eq-ai1}
 1-DR ,1-R D\in \Psi^{-N}(M)\rtimes G.
\end{equation} 
\end{lemma}
\begin{proof}
Since $D$ is elliptic, there exists an inverse symbol
$$
  \sigma_{pr}(D)^{-1}=\{r_g\}\in C^\infty(S^*M)\rtimes G.
$$
We define the almost inverse   $G$-operator $R_0$ by
$$
 R_0=\sum_g R_g\Phi_g, \quad \text{where } \sigma_{pr}(R_g)=r_g. 
$$
Then $1-R_0D,1-DR_0\in \Psi^{-1}(M)\rtimes G$, and  we set
$$
 R = (1+K_1+\ldots+K_1^{N-1}) R_0,\quad \text{where }K_1=1-R_0 D.
$$
\end{proof}

\begin{remark}
 In general,   inverses modulo smoothing operators   of elliptic $G$-operators can not be 
 represented as finite sums as in \eqref{gop1}. For  instance, consider the operator
 $D=1-\alpha B\Phi$, where $\alpha$ is a number, $B$  is a pseudodifferential operator of negative order and 
 $\Phi$ is an invertible quantized canonical transformation. This elliptic operator   
 is invertible for small $|\alpha| $ and the inverse is equal to the infinite sum
 $$
  D^{-1}=1+\alpha B\Phi+(\alpha B\Phi)^2+\ldots,
 $$
and in general can not be represented as a finite sum.
\end{remark}

\paragraph{\bf Analytic indices localized at conjugacy classes in $G$.}

Given an element $g\in G$, we define a linear functional
\begin{equation}\label{tau1}
 \begin{array}{ccc}
   \Tr_g: \Psi^{-N}(M)\rtimes G & \longrightarrow& \mathbb{C}\vspace{2mm}\\
   \sum\limits_l K_l\Phi_l & \longmapsto & \sum\limits_{l\in\langle g\rangle}\tr\left(K_l\Phi_l\right),
 \end{array}
\end{equation}
where $\langle g\rangle\subset G$ stands for the conjugacy class of $g$, and $\tr$ is the operator trace for operators in $L^2(M)$. 
The traces in \eqref{tau1} are defined whenever $N>\dim M$. 
A direct computation shows that $\Tr_g$ is a trace. Moreover, one has 
$$
 \Tr_g(AB)=\Tr_g(BA),\quad\text{for all } A,B\in \Psi(M)\rtimes G \text{ such that } \ord A+\ord B<-\dim M.
$$

\begin{definition}
Given an elliptic $G$-operator $D$, we define its {\em index localized at the conjugacy class $\langle g\rangle\subset G$} as
\begin{equation}\label{eq-indg1}
 \ind_g D=\Tr_g(1-R  D)-\Tr_g(1-DR )=\Tr_g[D,R]\in \mathbb{C},
\end{equation}
where $R $ is an almost-inverse element as in \eqref{eq-ai1} with $N>\dim M$.
\end{definition}
\begin{proposition}\label{prop6}
The localized index $\ind_g D$ is independent of the choice of the almost-inverse operator and therefore a well-defined   
 invariant of the complete symbol of $D$. It satisfies the following properties:
\begin{enumerate}
 \item[1)]  Consider $D$  as a Fredholm operator
 $$
  D:H^s(M) \longrightarrow H^{s-m}(M)
 $$
 for some $s$. Then its Fredholm index $\ind D$ is given by
 $$
  \ind D=\sum_{\langle g\rangle\subset G}\ind_g D,
 $$  
 where the sum is over all conjugacy classes in  $G$;
 \item[2)] We have  
 \begin{equation}\label{proj0}
  \ind_g D= \Tr_g(W_DP_0W_D^{-1}-P_0),
 \end{equation}
 for the invertible $W_D$ and projection $P_0$ defined as
 \begin{equation}\label{proj1}
   W_D=\left(%
          \begin{array}{cc}
             (2-DR)D  & 1-DR \\
             RD-1 & R \\
          \end{array}%
       \right)
,\qquad  P_0=\left(
           \begin{array}{cc}
             1  & 0 \\
             0  & 0          
           \end{array}
         \right),
 \end{equation}
  where $R$ is an almost-inverse operator such that \eqref{eq-ai1} holds.
   \item[3)] If $D_t$ is a smooth family with parameter $t\in[0,1]$ of elliptic operators and there exists a smooth family of operators $R_t$ such that
   \eqref{eq-ai1} holds, then $\ind_g D_t$ does not depend on $t$. 
\end{enumerate}
\end{proposition}
\begin{proof}
It is standard to prove that $\ind_g D$ is well defined, that is, it does not depend on the choice of $R$,  using the property that $\Tr_g$ is a trace.
 
Property 1) is obvious.

Property 2) follows by a direct computation. Indeed,  $W_D$ is invertible and the inverse is equal to
$$
W_D^{-1}=\left(%
          \begin{array}{cc}
             R & RD-1 \\
             1-DR & D(2-RD) \\
          \end{array}%
       \right).
$$
Then we calculate the right-hand side in \eqref{proj0} and obtain
$$
\Tr_g(W_DP_0W_D^{-1}-P_0)= \Tr_g((2-DR)DR+(RD-1)^2-1 )=
\Tr_g  [D,2R-RDR])=\ind_g D.
$$
Here we used the fact that $2R-RDR$ is an  almost-inverse for  $D$ and that $\ind_gD$ does not depend on the choice of  the almost-inverse operator. 

Let us now prove 3). We first note that the invertible element $W_{D_t}$ defined in \eqref{proj1} and the projection
$W_{D_t}P_0W_{D_t}^{-1}$ are also smooth in $t$. Hence, the trace $\Tr_g(W_{D_t}P_0W_{D_t}^{-1}-P_0)$ does not depend on $t$.
Thus, by 2) $\ind_g D_t$ also does not depend on $t$.
\end{proof}
\begin{remark}
The expressions in \eqref{proj0} and \eqref{proj1}  are just an explicit form of the boundary mapping in algebraic $K$-theory, see e.g.~\cite{Bla1,Nis3}). Note however, that we can not use the graph projection as in \cite{NeTs3}, since the algebraic crossed products we use are not spectrally invariant. 
\end{remark}

\paragraph{\bf Action of $G$ on semiclassical symbols.}

Now, given $g\in G$, we have a quantized canonical transformation $\Phi_g$. We denote the corresponding semiclassical
quantized canonical transformation as $\Phi_{g,h,\varepsilon,N}$.

Let us define the action of $g\in G$ on semiclassical symbols  $a\in \mathbb{A}'/ \mathbb{A}'_N$ by the formula 
\begin{equation}\label{eq-action1}
\varphi_{g,N}(a)=\sigma(\Phi_{g,h,\varepsilon,N} Op_h(a) \Phi_{g^{-1},h,\varepsilon,N})\in \mathbb{A}'/ \mathbb{A}'_N,
\end{equation}
where $\varepsilon=\varepsilon(a)$ is chosen as in Theorem~\ref{Egorov}. This element is well defined, since conjugation with
$\Phi_{g,h,\varepsilon,N}$ preserves the filtration of $\mathbb{A}$ (this follows from Eq.~\eqref{eq-symb4}).

Clearly, the element \eqref{eq-action1} will not change if we take a larger $N$ or a smaller $\varepsilon$.
Moreover, this element is independent of the choice of cut-off functions. Therefore, below we omit $\varepsilon$ for brevity. 
 
\begin{proposition}
\label{lem5}
\begin{enumerate}
\item The   mapping   
$$
 \begin{array}{ccc}
  \varphi_{g,N}:\mathbb{A}'/\mathbb{A}'_N & \longrightarrow & \mathbb{A}'/\mathbb{A}'_N\\
  a & \longmapsto  & \varphi_{g,N}(a)
\end{array} 
$$  
is an automorphism of the algebra $\mathbb{A}'/\mathbb{A}'_N$ and the collection of all such mappings for $g\in G$ defines
an action of $G$ on   $\mathbb{A}'/\mathbb{A}'_N$.
\item The actions $\varphi_{N}$ for different $N$ are compatible, i.e., the following diagram commutes:
\begin{equation}
 \xymatrix{
  \mathbb{A}'/\mathbb{A}'_{N+1} \ar[r]^{\varphi_{g,N+1}} \ar[d] & \mathbb{A}'/\mathbb{A}'_{N+1} \ar[d]\\
  \mathbb{A}'/\mathbb{A}'_N \ar[r]^{\varphi_{g,N}}        & \mathbb{A}'/\mathbb{A}'_{N}.  
 }
\end{equation}
\end{enumerate} 
\end{proposition}
\begin{proof}

1. Let us first prove that $\varphi_{g,N}$ is an automorphism. Indeed, given $a,a'\in\mathbb{A}'$, we have
equalities modulo $O(h^N)$-families
\begin{multline}
 Op_h(\varphi_{g,N}(a_1*a_2))=\Phi_{g,h,N}Op_h(a_1*a_2)\Phi_{g^{-1},h,N}\\
 =\Phi_{g,h,N}Op_h(a_1)Op_h(a_2)\Phi_{g^{-1},h,N}\\
 =\Phi_{g,h,N}Op_h(a_1)\Phi_{g^{-1},h,N} \Phi_{g,h,N} Op_h(a_2)\Phi_{g^{-1},h,N}\\
 =Op_h(\varphi_{g,N}(a_1 ))Op_h(\varphi_{g,N}(a_2))=Op_h(\varphi_{g,N}(a_1 )* \varphi_{g,N}(a_2)).
\end{multline}
Here the first equality is true by the definition of $\varphi_{g,N}$, the second and the last equalities are true, since 
$Op_h$ is a homomorphism, the third equation follows from Proposition~\ref{fio2}; the fourth equality is again just the definition of $\varphi_{g,N}$. 

2. Let us now prove that the collection $\{\varphi_{g,N}\}_{g\in G}$ defines  a group action. 
Given $g_1,g_2\in G$,  we have to show that
\begin{equation}\label{eq-12}
 \varphi_{g_1,N}(\varphi_{g_2,N}(a))=\varphi_{g_1g_2,N}(a).
\end{equation}
Indeed, the left hand side of this equality is the symbol of the composition
\begin{equation}\label{eq-12a}
 \Phi_{g_1,h,N}  \Phi_{g_2,h,N} Op_h(a) \Phi_{g^{-1}_2,h,N} \Phi_{g^{-1}_1,h,N}. 
\end{equation}
However, by Proposition~\ref{fio2} we have  
$$
 \Phi_{g_1,h,N}  \Phi_{g_2,h,N} Op_h(a)=\Phi_{g_1g_2,h,N}   Op_h(a)\quad \mod O(h^N)\text{-families}
$$
and also
$$
Op_h(a) \Phi_{g^{-1}_2,h,N} \Phi_{g^{-1}_1,h,N}=Op_h(a) \Phi_{g^{-1}_2g^{-1}_1,h,N}\quad \mod O(h^N)\text{-families}. 
$$
Hence, the left hand side in \eqref{eq-12a} is equal to 
$$
 \Phi_{g_1g_2,h,N}   Op_h(a)\Phi_{g^{-1}_2g^{-1}_1,h,N} \quad \mod O(h^N)\text{-families},
$$
which is just the right hand side in \eqref{eq-12}. Hence, we showed that 
$\{\varphi_{g,N}\}_{g\in G}$ is a representation of $G$ by automorphisms of $\mathbb{A}'/\mathbb{A}'_{N}$.

3. Compatibility of representations $\varphi_{g,N}$ for different $N$ follows from the 
fact that the element \eqref{eq-action1}  does not change, if we take larger $N$ and smaller $\varepsilon$ on the right hand side.  
\end{proof}

\begin{corollary}
The formula  
\begin{equation}\label{eq-op5}
g\in G, a\in\mathbb{A}'\quad \longmapsto \quad  \varphi_g(a)=\Bigl[1+\sum_{k\ge 1, 0<|\alpha|+|\beta|\le 2k}h^k \mu_{g,k,\alpha,\beta} D^\alpha_x D^\beta_\xi \Bigr]{C^{-1}_g}^*a,
\end{equation}
where the coefficients $\mu_{g,k,\alpha,\beta}(x,\xi)$ were defined in \eqref{eq-symb4}, defines an action of $G$ on  $\mathbb{A}'$.
\end{corollary}
Hence, we can define the  algebraic crossed product $\mathbb{A}' \rtimes G$. Its elements are called
{\em semiclassical $G$-symbols}. We denote the product of two semiclassical $G$-symbols by ``*''.

\paragraph{\bf Ellipticity and parametrix construction.}

$G$ acts on $T^*_0M$ by homogeneous canonical transformations. Thus, we have crossed products
$S^j(T^*_0M)\rtimes G$ of spaces of symbols of order $j$, which vanish in a neighborhood of the zero section. 

\begin{definition}
  A symbol $a\in \mathbb{A}'\rtimes G$ of order zero is {\em elliptic} if its leading symbol $a_0\in S^0(T^*_0M )\rtimes G$ is   invertible modulo symbols of order $-1$, i.e., there exists a classical
  leading symbol $r_0\in S^{0}(T^*_0M)\rtimes G$ such that the differences
 \begin{equation}\label{eq-par5}
 a_0r_0-1,\quad  r_0a_0-1 \quad \text{ are of order }\le -1.
\end{equation} 
\end{definition} 

Denote by $\mathbb{B}$ the algebra $\mathbb{A}'\rtimes G$ with adjoint unit. Its elements can be represented
as collections 
\begin{equation}\label{coll1}
    \left\{\sum_{j\ge 0}h^j a_{l,j}(x,\xi)\right\}_{l\in G},
\end{equation}
where the coefficients $a_{l,j}(x,\xi)$ satisfy all the properties for elements in $\mathbb{A}'$, except that $a_{e,0}(x,\xi)$ is allowed 
to be equal to a nonzero constant in a neighborhood of the zero section in $T^*M$. Denote by $\mathbb{B}_{N}\subset \mathbb{B}$
the ideal of elements \eqref{coll1}, whose coefficients $a_{l,j}(x,\xi)$ have order 
$\le -N$. 

\begin{lemma}
Let   $a\in \mathbb{B}$ be an elliptic symbol of order zero. Then for each $N\ge 1$ there exists a symbol $r_N \in \mathbb{B}$ such that  
\begin{equation}\label{eq-ai2}
 1-a*r_N,1-r_N*a  \in \mathbb{B}_{N}.
\end{equation} 
\end{lemma}

\begin{proof}
Since $a$ is elliptic, there exists $r_0\in S^{0}(T^*_0M)\rtimes G$ such that \eqref{eq-par5} holds.
Clearly, $r_0\in \mathbb{B}$ and we have
\begin{equation*}
 a*r_0=\left(a_0+\sum_{1\le j\le N}h^ja_j\right)*r_0=a_0r_0+\sum_{1\le j\le N}h^jb_j \equiv 1-w, 
\end{equation*}
where  
$$
 w=(1-a_0r_0)-\sum_{1\le j\le N}h^jb_j \in \mathbb{B}_{1}.
$$
Hence, we obtain
$$
a*r_0*(1+w+w*w+...+w^N)=(1-w)*(1+w+w*w+...+w^N)=1-w^{N+1},
$$
where $w^{N+1}\in \mathbb{B}_{N+1}.$
Therefore, if we set
$ 
r=r_0*(1+w+w*w+...+w^N),
$ 
then we have 
$$
 a*r-1=-w^{N+1}\in \mathbb{B}_{N+1}.
$$ 
A computation shows also that $r*a-1\in \mathbb{B}_{N+1}$. Hence,  \eqref{eq-ai2} is valid. 
\end{proof}

\paragraph{\bf Localized traces of  $G$-operators.}

Given a finite order element $g\in G$ and a number $N>2\dim M$, we  now   define a linear functional on elements
$$
 a=\{a_l\}_{l\in G}\in \mathbb{A}' \rtimes G\qquad \text{such that } \ord a<-2\dim M.
$$
The desired functional, denoted by $\tau_{g,N}$, is defined as
\begin{equation}\label{eq-6}
   \tau_{g,N}(a)= \sum\limits_{l\in\langle g\rangle}  \tr(Op_h(a_{l})\Phi_{l,h,\varepsilon,N})\in 
   \left(h^{-\dim T^*M^g/2} \mathbb{C}[h]\right)/h^{N-\dim M}, 
\end{equation}
where $\left(h^{-\dim T^*M^g/2} \mathbb{C}[h]\right)/h^{N-\dim M}$ stands for the space of  Laurent
polynomials 
$$
\sum_{-\dim T^*M^g/2\le j<N-\dim M}c_j h^j
$$
and $\varepsilon$ in  $\Phi_{l,h,\varepsilon,N}$  is chosen such  that the first $N$ components in the expansion of $a_{l}\in \mathbb{A}'$
in powers of $h$ are equal to zero on the set 
$C_l\{|\xi|<2\varepsilon\}\subset T^*M$.

We claim that  $\tau_{g,N}(a)$ in \eqref{eq-6} is independent of the choice of $\varepsilon$ and $\Phi_{l,h,\varepsilon,N}$. 
Indeed, a different choice of these data  gives an operator family, which differs from the original family by an $O(h^N)$-family according to Proposition~\ref{fio1}. 
As the trace of such a family is $O(h^{N-\dim M})$ by Lemma~\ref{lemma-trace}, it is equal to zero in the quotient \eqref{eq-6}. 
Moreover, the traces for different $N$ are compatible:  
$$
 \tau_{g,N+1}(a)\equiv \tau_{g,N}(a)\mod h^{N-\dim M}. 
$$ 
Hence, in the limit $N\to\infty$ these functionals assemble in a functional denoted by
\begin{equation}\label{eq-6q}
  \tau_g: \mathbb{A}'\rtimes G \longrightarrow  h^{-\dim T^*M^g/2}\mathbb{C}[[h]]
\end{equation}
defined on symbols of  order $<-2\dim M$. Here $\mathbb{C}[[h]]$ stands for the algebra of formal power series in $h$.

\begin{proposition}\label{pro77}
The functional \eqref{eq-6q} is   a trace:   
Given $a,b\in \mathbb{A}' \rtimes G$ such that $\ord a+\ord b<-2\dim M$, we have
$$
\tau_{g}(a*b)=\tau_{g}(b*a).
$$
\end{proposition}
\begin{proof}
It suffices to prove the trace property for $a$ and $b$ with only one nonzero component denoted by $a_l$ and $b_k$,
where $lk\in \langle g\rangle.$  Then $a*b$  also has only one nonzero component denoted by $(a*b)_{lk}$.

By the definitions of the star product and the crossed product 
and Theorem~\ref{Egorov}, we have
$$
(a*b)_{lk}\equiv a_l*\sigma(\Phi_{l,\varepsilon,N}Op_h(b_k)\Phi^{-1}_{l,\varepsilon,N})
$$
modulo symbols of $O(h^N)$-families. Hence, we obtain for the corresponding operators
\begin{equation}\label{oppa1}
Op_h\bigl[(a*b)_{lk}\bigr]\equiv Op_h(a_l)\Phi_{l,\varepsilon,N}Op_h(b_k)\Phi^{-1}_{l,\varepsilon,N}\mod O(h^N)\text{-families}. 
\end{equation}
 
We can now compute the trace functionals:
\begin{multline}
 \tau_{g}(a*b)\equiv \tr(Op_h\bigl[(a*b)_{lk}\bigr]\Phi_{lk,\varepsilon,N})\equiv
 \tr(Op_h(a_l)\Phi_{l,\varepsilon,N}Op_h(b_k)\Phi^{-1}_{l,\varepsilon,N}\Phi_{lk,\varepsilon,N}) \\
 \equiv \tr(Op_h(a_l)\Phi_{l,\varepsilon,N}Op_h(b_k) \Phi_{k,\varepsilon,N})\equiv
 \tr(Op_h(b_k) \Phi_{k,\varepsilon,N}Op_h(a_l)\Phi_{l,\varepsilon,N})\equiv \tau_{g}(b*a).
\end{multline}
Here all comparisons are modulo $h^{N-\dim M}$: The first is the definition of $\tau_g$,
the second follows from \eqref{oppa1}, the third follows from the composition formula in Proposition~\ref{fio2}, 
the fourth follows since the operator trace has the trace property, while the final comparison follows from all previous
comparisons, if we consider $b*a$ instead of $a*b$.
This completes the proof. 
\end{proof}

\paragraph{\bf Algebraic indices.} Let $g\in G$ be an element of finite order. 

\begin{definition}\label{algind1}
Given an  elliptic symbol $a\in \mathbb{B}$, its {\em algebraic index localized at the conjugacy
class} $\langle g\rangle\subset G$  is defined as
\begin{equation}\label{eq-algind4}
 \widetilde{\ind}_{g,N} a=\tau_{g}(1-r_N*a)-\tau_{g}(1-a*r_N)=\tau_{g}[a,r_N]\in \left(h^{-\dim T^*M^g/2} \mathbb{C}[h]\right)/h^{N-\dim M} ,
\end{equation}
where $r$ is an almost-inverse symbol for $a$ such that \eqref{eq-ai2} holds.
\end{definition}
The algebraic index \eqref{eq-algind4} is independent of the choice of the almost-inverse symbol $r_N$
and the algebraic indices for different $N$ are compatible
$$
 \widetilde{\ind}_{g,N} a\equiv \widetilde{\ind}_{g,N+1}a \mod h^{N-\dim M}.
$$
They  define the algebraic index as $N\to \infty$
\begin{equation}\label{eq-algind4a}
 \widetilde{\ind}_{g} a\in h^{-\dim T^*M^g/2} \mathbb{C}[[h]].
\end{equation}

\begin{proposition}\label{p-5}
We have
 \begin{equation}\label{proj0a}
  \widetilde\ind_g a\equiv \tau_{g}(w_a*p_0*w_a^{-1}-p_0)\mod h^{N-\dim M},
 \end{equation}
 where 
 \begin{equation}\label{proj1a}
   w_a=\left(%
          \begin{array}{cc}
             (2-a*r_N)*a  & 1-a*r_N \\
             r_N*a-1 & r_N \\
          \end{array}%
       \right)
,\quad  p_0=\left(
           \begin{array}{cc}
             1  & 0 \\
             0  & 0          
           \end{array}
         \right),
 \end{equation}
  while $r_N$ is an almost-inverse symbol as in \eqref{eq-ai2}.
\end{proposition}
\begin{proof}
Property \eqref{proj0a} follows by a direct computation. Indeed,  $w_a$ is invertible 
in $\mathbb{B}$ and the inverse is equal to
$$
w_a^{-1}=\left(%
          \begin{array}{cc}
             r_N & r_N*a-1 \\
             1-a*r_N & a*(2-r_N*a) \\
          \end{array}%
       \right).
$$
Then we calculate the right-hand side in \eqref{proj0a} and obtain
\begin{multline*}
\tau_{g}(w_a*p_0*w_a^{-1}-p_0)= \tau_{g}((2 -a*r_N)*a*r_N+(r_N*a-1)^2-1 )=\\
= \tau_{g}  [a,2r_N-r_N*a*r_N] \equiv \widetilde\ind_g a \mod h^{N-\dim M}.
\end{multline*}
Here we used the fact that $2r_N-r_N*a*r_N$ is an  almost-inverse for  $a$ and that $\widetilde \ind_g a$ does not depend on the choice of the  almost-inverse symbol. 
\end{proof}

\paragraph{\bf Analytic and algebraic indices are equal.}
  
Given an elliptic symbol $a\in\mathbb{B}$, we define the semiclassical $G$-operator
  (cf.~\eqref{eq-6})
$$
Op_h(a)=\sum_{l\in G} Op_h(a_{l})\Phi_{l,h,\varepsilon,N}.
$$ 
This Fredholm family is constant in $h$ modulo infinitely smoothing operators. Hence, 
its analytic index $\ind_g Op_h(a)$ (see \eqref{eq-indg1}) is constant in $h$ by Proposition~\ref{prop6}. On the other hand, our elliptic symbol 
$a$ has algebraic index localized at $g$
$$
 \widetilde{\ind}_g a\in h^{-\dim T^*M^g/2} \mathbb{C}[[h]].
$$

The following theorem is the main result of this paper.
\begin{theorem}\label{th-1}
Given a finite order element $g\in G$, the algebraic index localized at $g$ has no negative and no positive powers of $h$, and its constant term is equal to the analytic index:
\begin{equation}\label{eq-7}
  \ind_g Op_h(a)=\left(\widetilde{\ind}_g a\right)\Bigr|_{h=0}.
\end{equation}
\end{theorem}

\begin{proof}
Given $N>2\dim M$, we consider the element  (see \eqref{proj1a})
$$
w_a=\left(%
          \begin{array}{cc}
             (2-a*r_N)*a  & 1-a*r_N\\
             r_N*a-1 & r_N \\
          \end{array}%
       \right)\in \mathbb{B},
$$
where an almost inverse symbol $r_N$ is chosen such that \eqref{eq-ai2} holds.

Denote for brevity $A=Op_h(a)$ and similarly $R=Op_h(r_N)$.
Then the analytic index of $A$ is independent of $h$, and by  Proposition~\ref{prop6}  we have for the analytic index
\begin{equation}\label{eq-7a}
 \ind_g A=\Tr_g(W_AP_0 W_A^{-1}-P_0), 
\end{equation}
where (cf.~\eqref{proj1})
$$
 W_A=\left(%
          \begin{array}{cc}
             (2-AR)A  & 1-AR \\
             RA-1 & R \\
          \end{array}%
       \right).
$$
By the definition of the $*$-product in $\mathbb{B}$ the difference
$$
 W_AP_0 W_A^{-1}- Op_h(w_a*p_0*w_a^{-1})
$$
is an $O(h^N)$-family. Further, we obtain the comparison
\begin{equation}\label{eq-7b}
 \Tr_g(W_AP_0 W_A^{-1}-P_0)\equiv \tau_{g}(w_a*p_0*w_a^{-1}-p_0) \mod h^{N-\dim M}.
\end{equation}
by the definition of  the trace $\tau_{g}$ in \eqref{eq-6} and Lemma~\ref{lemma-trace}.
Finally, the right hand side in \eqref{eq-7b} is equal to the algebraic index $\mod  h^{N-\dim M}$ by Proposition~\ref{p-5}:
\begin{equation}\label{eq-7c}
 \tau_{g}(w_a*p_0*w_a^{-1}-p_0)\equiv  \widetilde{\ind}_g a \mod h^{N-\dim M}.
\end{equation}
Therefore, equalities \eqref{eq-7a}, \eqref{eq-7b} and \eqref{eq-7c} imply that the algebraic index has only the constant term and is equal to the
analytic index, i.e., we obtain the desired formula \eqref{eq-7}.

This ends the proof of Theorem~\ref{th-1}.
\end{proof}

\paragraph{\bf The Fredholm index.}

Theorem~\ref{th-1} treats indices localized at torsion elements of the group.
It turns out that for some infinite order elements of the group the localized index is always equal to zero.
More precisely, the following vanishing result holds, cf.~Proposition~9.4 in~\cite{NaSaSt17}.
\begin{proposition}\label{prop7}
 Given an elliptic operator $D$ and $g_0\in G$, we have
 $$
  \ind_{g_0} D=0, 
 $$
 whenever there exists a group homomorphism $\chi:G\to \mathbb{Z}$  such that $\chi(g_0)\ne 0.$
 %ES $g_0\in G$ satisfies the following property:   There exists a homomorphism of groups 
\end{proposition}

\begin{proof}
 We use $\chi$ to define the family of unital automorphisms 
 $$
 U_t\in {\rm Aut}(\Psi(M)\rtimes G),\quad U_t\{D_g\}=\{e^{it\chi(g )}D_g\}, \quad t\in [0,2\pi].
 $$
Given 
an elliptic operator $D$ with almost-inverse $R$, the homotopies $D_t=U_t(D), R_t=U_t(R)$
satisfy the assumptions of Proposition~\ref{prop6}, and hence the index of this family does not depend on $t$:
\begin{equation}\label{eq81}
\ind_{g_0} D_t=\ind_{g_0} D.
\end{equation}
On the other hand, by the definition of the localized index we have
 \begin{eqnarray}\label{eq82}
\lefteqn{\ind_{g_0} D_t=\Tr_{g_0} [D_t,R_t]\nonumber}\\
&=&\Tr_{g_0} U_t([D,R])=e^{it\chi(g_0)} \Tr_{g_0} [D,R]=e^{it\chi(g_0)}\ind_{g_0} D.
 \end{eqnarray}
 It now follows from \eqref{eq81} and \eqref{eq82} and our condition $\chi(g_0)\ne 0$ that $\ind_{g_0} D=0$.
\end{proof}
The conditions of this proposition are satisfied for  all infinite-order elements in finite extensions of Abelian groups (in particular, for all finite groups and Abelian groups).
Hence, we obtain  the following corollary from Propositions~\ref{prop6},~\ref{prop7} and Theorem~\ref{th-1}.
\begin{corollary}
 Given an elliptic symbol $a\in C^\infty(S^*M)\rtimes  G$, where $G$ is a finite extension of an Abelian group, the Fredholm index of the corresponding $G$-operator denoted by $A$
 is equal to the sum of localized algebraic indices over torsion conjugacy classes in $G$:
\begin{equation}\label{eq83}
\ind A= \sum_{\langle g\rangle\subset {\rm Tor}\;G} \left(\widetilde{\ind}_g a\right)\Bigr|_{h=0}.
\end{equation}
Here ${\rm Tor}\;G$ is the torsion subgroup of $G$.
\end{corollary}
\begin{remark}
 Formula~\eqref{eq83} also holds for torsion free groups of polynomial growth.  The proof follows from the equality $\ind D=\ind_e D$,
 which can be obtained as in~\cite{SaSt23}.
\end{remark}

\end{document}